\documentclass[a4paper]{article}
\usepackage{mathrsfs, amsmath, amsthm, amssymb, bm, mathtools, thm-restate}
\usepackage{graphicx, xcolor, tikz}
\usetikzlibrary{arrows.meta, intersections, calc, shapes}
\usepackage{caption, subcaption, ulem}
\usepackage{array}
\usepackage{cases}
\makeatletter
\def\th@plain{%
  \upshape 
}
\makeatother

\makeatletter
\renewenvironment{proof}[1][\proofname]{\par
  \pushQED{\qed}%
  \normalfont \topsep6\p@\@plus6\p@\relax
  \trivlist
  \item[\hskip\labelsep
        \bfseries
    #1\@addpunct{.}]\ignorespaces
}{%
  \popQED\endtrivlist\@endpefalse
}
\makeatother

\newtheorem{theorem}{Theorem}[section]
\newtheorem{lemma}{Lemma}[section]
\newtheorem{corollary}[theorem]{Corollary}
\theoremstyle{definition}
\newtheorem{definition}{Definition}
\newtheorem{remark}{Remark}
\newtheorem{problem}{Problem}

\usepackage[left=25mm,top=25mm,bottom=25mm,right=25mm]{geometry}
\usepackage[T1]{fontenc}
\usepackage[inline]{enumitem}
\usepackage[square, numbers, sort&compress]{natbib}

\usepackage[
  bookmarks=true,%
  bookmarksnumbered=true, 
  bookmarksopen=true, 
  plainpages=false,%
  pdfpagelabels,%
  colorlinks=true, 
  linkcolor=black, 
  citecolor=black,%
  anchorcolor=green,
  urlcolor= blue,
  breaklinks=true,
  hyperindex=true]{hyperref}

\newcommand{\etal}{et~al.\ }
\newcommand{\ie}{i.e.,\ }
\def\mad(#1){\mathrm{mad}(#1)}
\def\int(#1){\mathrm{int}(#1)}
\def\ext(#1){\mathrm{ext}(#1)}
\def\Int(#1){\mathrm{Int}(#1)}
\def\Ext(#1){\mathrm{Ext}(#1)}

\begin{document}
\title{Planar graphs without normally adjacent short cycles}
\author{Fangyao Lu\textsuperscript{a}\quad Mengjiao Rao\textsuperscript{b}\quad Qianqian Wang\textsuperscript{a}\quad Tao Wang\textsuperscript{a,}\footnote{{\tt Corresponding author: wangtao@henu.edu.cn}}\\
{\small \textsuperscript{a}Center for Applied Mathematics, Henan University, }\\
{\small Kaifeng, 475004, P. R. China}\\
{\small \textsuperscript{b}Center for Discrete Mathematics, Fuzhou University, }\\
{\small Fujian, 350003, P. R. China}}
\date{}
\maketitle
\begin{abstract}
Let $\mathscr{G}$ be the class of plane graphs without triangles normally adjacent to $8^{-}$-cycles, without $4$-cycles normally adjacent to $6^{-}$-cycles, and without normally adjacent $5$-cycles. In this paper, it is shown that every graph in $\mathscr{G}$ is $3$-choosable. Instead of proving this result, we directly prove a stronger result in the form of ``weakly'' DP-$3$-coloring. The main theorem improves the results in [J. Combin. Theory Ser. B 129 (2018) 38--54; European J. Combin. 82 (2019) 102995]. Consequently, every planar graph without $4$-, $6$-, $8$-cycles is $3$-choosable, and every planar graph without $4$-, $5$-, $7$-, $8$-cycles is $3$-choosable. In the third section, using almost the same technique, we prove that the vertex set of every graph in $\mathscr{G}$ can be partitioned into an independent set and a set that induces a forest, which strengthens the result in [Discrete Appl. Math. 284 (2020) 626--630]. In the final section, tightness is discussed.

Keywords: DP-coloring; Near-bipartite; IF-coloring; List coloring
\end{abstract}

\section{Introduction}
\label{sec1}
All graphs in this paper are finite, undirected and simple. For a graph $G$, a \textbf{list-assignment} $L$ assigns to each vertex $v$ a set $L(v)$ of colors available at $v$. An \textbf{$L$-coloring} of $G$ is a proper coloring $\phi$ of $G$ such that $\phi(v) \in L(v)$ for all $v \in V(G)$. A \textbf{list $k$-assignment} $L$ is a list-assignment such that $|L(v)| \geq k$ for all $v \in V(G)$. A graph $G$ is \textbf{$k$-choosable} or \textbf{list $k$-colorable} if it has an $L$-coloring for any list $k$-assignment $L$. The \textbf{list chromatic number} or \textbf{choice number} $\chi_{\ell}(G)$ is the least integer $k$ such that $G$ is $k$-choosable. 

The Four Color Theorem states that every planar graph is $4$-colorable. Gr\"{o}tzsch \cite{MR0116320} showed that every planar graph without triangles is $3$-colorable. Much more sufficient conditions for $3$-colorability and $3$-choosability are extensively studied. Thomassen \cite{MR1328294} showed that every planar graph with girth at least five is $3$-choosable. Borodin \cite{MR3004485} conjectured that every planar graph without cycles of length $4$ to $8$ is $3$-choosable. 

A widely used technique in ordinary vertex coloring is the identification of vertices, but this is not feasible in general for list-coloring because different vertices may have different lists. To overcome this difficulty, Dvo\v{r}\'{a}k and Postle \cite{MR3758240} introduced DP-coloring, also called correspondence coloring, as a generalization of list-coloring. 

\begin{definition}\label{DEF1}
Let $G$ be a simple graph and $L$ be a list-assignment for $G$. For each vertex $v \in V(G)$, let $L_{v} = \{v\} \times L(v)$; for each edge $uv \in E(G)$, let $\mathscr{M}_{uv}$ be a matching between the sets $L_{u}$ and $L_{v}$, and let $\mathscr{M} := \bigcup_{uv \in E(G)}\mathscr{M}_{uv}$. We call $\mathscr{M}$ a \textbf{matching assignment}. The matching assignment is called a \textbf{$k$-matching assignment} if $L(v) = [k]$ for each $v \in V(G)$. A \textbf{cover} of $G$ is a graph $H_{L, \mathscr{M}}$ (simply write $H$) satisfying the following two conditions: 
\begin{enumerate}[label =(C\arabic*)]
\item the vertex set of $H$ is the disjoint union of $L_{v}$ for all $v \in V(G)$; 
\item the edge set of $H$ is the matching assignment $\mathscr{M}$.
\end{enumerate}
\end{definition}
Note that the matching $\mathscr{M}_{uv}$ is not required to be a perfect matching between the sets $L_{u}$ and $L_{v}$, and possibly it is empty. The induced subgraph $H[L_{v}]$ is an independent set for each vertex $v \in V(G)$. 

\begin{definition}
Let $G$ be a simple graph and $H$ be a cover of $G$. An \textbf{$\mathscr{M}$-coloring} of $G$ is an independent set $\mathcal{I}$ in $H$ such that $|\mathcal{I} \cap L_{v}| = 1$ for each vertex $v \in V(G)$. The graph $G$ is \textbf{DP-$k$-colorable} if for any list-assignment $L(v) \supseteq [k]$ and any matching assignment $\mathscr{M}$, it has an $\mathscr{M}$-coloring. The \textbf{DP-chromatic number $\chi_{\mathrm{DP}}(G)$} of $G$ is the least integer $k$ such that $G$ is DP-$k$-colorable. 
\end{definition}

DP-coloring is quite different from list-coloring, for example each even cycle is $2$-choosable but it is not DP-$2$-colorable. Dvo\v{r}\'{a}k and Postle gave a relation between DP-coloring and list-coloring. 

Let $W = w_{1}w_{2}\dots w_{m}$ with $w_{m} = w_{1}$ be a closed walk of length $m - 1$ in $G$, a matching assignment is \textbf{inconsistent} on $W$, if there exist $c_{1}, \dots, c_{m}$ such that $c_{i} \in L(w_{i})$ for $i \in [m]$ and $(w_{i}, c_{i})(w_{i+1}, c_{i+1})$ is an edge in $\mathscr{M}_{w_{i}w_{i+1}}$ for $i \in [m-1]$ and $c_{1} \neq c_{m}$. Otherwise, the matching assignment is \textbf{consistent} on $W$. We say that a matching assignment is \textbf{consistent} if it is consistent on every closed walk in $G$. 

\begin{theorem}[Dvo\v{r}\'{a}k and Postle \cite{MR3758240}]
A graph $G$ is $k$-choosable if and only if $G$ is $\mathscr{M}$-colorable for every consistent $k$-matching assignment $\mathscr{M}$. 
\end{theorem}

With the aid of DP-coloring, Dvo\v{r}\'{a}k and Postle \cite{MR3758240} solved the longstanding conjecture by Borodin. 
\begin{theorem}[Dvo\v{r}\'{a}k and Postle \cite{MR3758240}]\label{3C}
Every planar graph without cycles of length $4$ to $8$ is $3$-choosable. 
\end{theorem}

An edge $uv$ in $G$ is \textbf{straight} in a $k$-matching assignment $\mathscr{M}$ if $(u, c_{1})(v, c_{2}) \in \mathscr{M}_{uv}$ satisfies $c_{1} = c_{2}$. An edge $uv$ in $G$ is \textbf{full} in a $k$-matching assignment $\mathscr{M}$ if $\mathscr{M}_{uv}$ is a perfect matching. 
\begin{lemma}[Dvo\v{r}\'{a}k and Postle \cite{MR3758240}]\label{ST}
Let $G$ be a graph with a $k$-matching assignment $\mathscr{M}$, and let $K$ be a subgraph of $G$. If for every cycle $\mathcal{Q}$ in $K$, the assignment $\mathscr{M}$ is consistent on $\mathcal{Q}$ and all edges of $\mathcal{Q}$ are full, then we may rename $L(u)$ for $u \in V(K)$ to obtain a $k$-matching assignment $\mathscr{M}'$ for $G$ such that all edges of $K$ are straight in $\mathscr{M}'$. 
\end{lemma}

In order to prove \autoref{3C}, they showed a stronger result as the following. 
\begin{theorem}[Dvo\v{r}\'{a}k and Postle \cite{MR3758240}]\label{N4-8}
Let $G$ be a plane graph without cycles of length $4$ to $8$. Let $S$ be a set of vertices of $G$ such that $|S| \leq 1$ or $S$ consists of all vertices on a face of $G$. Let $\mathscr{M}$ be a $3$-matching assignment for $G$ such that $\mathscr{M}$ is consistent on every closed walk of length three in $G$. If $|S| \leq 12$, then every $\mathscr{M}$-coloring $\phi$ of $G[S]$ can be extended to an $\mathscr{M}$-coloring $\varphi$ of $G$. 
\end{theorem}

Two cycles are \textbf{adjacent} if they have at least one common edge. An $\ell_{1}$-cycle and an $\ell_{2}$-cycle are \textbf{normally adjacent} if they form an $(\ell_{1} + \ell_{2} - 2)$-cycle with exactly one chord. In other words, two cycles are normally adjacent if their intersection is $K_{2}$. Recently, Liu and Li \cite{MR3983123} improved \autoref{N4-8} to the following result by allowing cycles of length $4$ to $8$ but forbidding adjacent cycles of length at most $8$. 

\begin{theorem}[Liu and Li \cite{MR3983123}]\label{NON}
Let $G$ be a plane graph without adjacent cycles of length at most $8$. Let $S$ be a set of vertices of $G$ such that $|S| \leq 1$ or $S$ consists of all vertices on a face of $G$. Let $\mathscr{M}$ be a $3$-matching assignment for $G$ such that $\mathscr{M}$ is consistent on every closed walk of length three in $G$. If $|S| \leq 12$, then every $\mathscr{M}$-coloring $\phi$ of $G[S]$ can be extended to an $\mathscr{M}$-coloring $\varphi$ of $G$. 
\end{theorem}

This implies the $3$-choosability of planar graphs without adjacent cycles of length at most $8$. 
\begin{theorem}[Liu and Li \cite{MR3983123}]\label{adj8-}
Every planar graph without adjacent cycles of length at most $8$ is $3$-choosable. 
\end{theorem}

The first goal of this paper is to further improve \autoref{NON} to the following result by allowing adjacent cycles of length $6$ to $8$ and changing the condition on precolored vertices from faces to cycles. But before we state the main theorem, it's necessary to give a new concept. A cycle is \textbf{abnormal} if it is the $11$- or $12$-cycle in a subgraph isomorphic to a configuration in \autoref{abnormal}. A cycle is \textbf{normal} if it is not an abnormal cycle. A $d$-vertex, $d^{+}$-vertex or $d^{-}$-vertex is a vertex of degree $d$, at least $d$, or at most $d$ respectively. Similar definitions can be applied to faces and cycles. 

\begin{figure}%
\centering
\def\s{1}
\subcaptionbox{\label{abnormala}}[0.2\linewidth]{\begin{tikzpicture}
\draw [line width = 1.5pt, blue] (0:\s)--(360/11:\s)--(360*2/11:\s)--(360*3/11:\s)--(360*4/11:\s)--(360*5/11:\s)--(360*6/11:\s)--(360*7/11:\s)--(360*8/11:\s)--(360*9/11:\s)--(360*10/11:\s)--cycle;
\foreach \ang in {0, 360/11, 360*2/11, 360*3/11, 360*4/11, 360*5/11, 360*6/11, 360*7/11, 360*8/11, 360*9/11, 360*10/11}{\fill 
(\ang:\s) circle (2pt);}
\draw (0, 0)--(0:\s);
\draw (0, 0)--(360*3/11:\s);
\draw (0, 0)--(360*7/11:\s);
\fill (0, 0) circle (2pt);
\node at (360*1.5/11:0.5*\s) {$5$};
\node at (360*5/11:0.5*\s) {$6$};
\node at (360*9/11:0.5*\s) {$6$};
\end{tikzpicture}}
\subcaptionbox{\label{abnormalb}}[0.2\linewidth]{\begin{tikzpicture}
\draw [line width = 1.5pt, blue] (0:\s)--(30:\s)--(60:\s)--(90:\s)--(120:\s)--(150:\s)--(180:\s)--(210:\s)--(240:\s)--(270:\s)--(300:\s)--(330:\s)--cycle;
\foreach \ang in {0, 30, 60, 90, 120, 150, 180, 210, 240, 270, 300, 330}{\fill 
(\ang:\s) circle (2pt);}
\draw (0, 0)--(0:\s);
\draw (0, 0)--(90:\s);
\draw (0, 0)--(210:\s);
\fill (0, 0) circle (2pt);
\node at (45:0.5*\s) {$5$};
\node at (150:0.5*\s) {$6$};
\node at (30*9.5:0.5*\s) {$7$};
\end{tikzpicture}}
\subcaptionbox{\label{abnormalc}}[0.2\linewidth]{\begin{tikzpicture}
\draw [line width = 1.5pt, blue] (0:\s)--(30:\s)--(60:\s)--(90:\s)--(120:\s)--(150:\s)--(180:\s)--(210:\s)--(240:\s)--(270:\s)--(300:\s)--(330:\s)--cycle;
\foreach \ang in {0, 30, 60, 90, 120, 150, 180, 210, 240, 270, 300, 330}{\fill 
(\ang:\s) circle (2pt);}
\draw (0, 0)--(0:\s);
\draw (0, 0)--(120:\s);
\draw (0, 0)--(240:\s);
\fill (0, 0) circle (2pt);
\node at (60:0.5*\s) {$6$};
\node at (180:0.5*\s) {$6$};
\node at (300:0.5*\s) {$6$};
\end{tikzpicture}}
\subcaptionbox{\label{abnormald}}[0.2\linewidth]{\begin{tikzpicture}
\draw [line width = 1.5pt, blue] (15:\s)--(45:\s)--(75:\s)--(105:\s)--(135:\s)--(165:\s)--(195:\s)--(225:\s)--(255:\s)--(285:\s)--(315:\s)--(345:\s)--cycle;
\foreach \ang in {15, 45, 75, 105, 135, 165, 195, 225, 255, 285, 315, 345}{\fill 
(\ang:\s) circle (2pt);}
\draw (0.4*\s, 0)--(45:\s);
\draw (0.4*\s, 0)--(315:\s);
\draw (-0.4*\s, 0)--(135:\s);
\draw (-0.4*\s, 0)--(225:\s);
\draw (-0.4*\s, 0)--(0.4*\s, 0);
\fill (0.4*\s, 0) circle (2pt);
\fill (-0.4*\s, 0) circle (2pt);
\node at (0:0.7*\s) {$5$};
\node at (180:0.7*\s) {$5$};
\node at (90:0.5*\s) {$6$};
\node at (270:0.5*\s) {$6$};
\end{tikzpicture}}
\subcaptionbox{\label{abnormale}}[0.2\linewidth]{\begin{tikzpicture}
\draw [line width = 1.5pt, blue] (0:\s)--(30:\s)--(60:\s)--(90:\s)--(120:\s)--(150:\s)--(180:\s)--(210:\s)--(240:\s)--(270:\s)--(300:\s)--(330:\s)--cycle;
\foreach \ang in {0, 30, 60, 90, 120, 150, 180, 210, 240, 270, 300, 330}{\fill 
(\ang:\s) circle (2pt);}
\draw (0, 0)--(0:\s);
\draw (0, 0)--(150:\s);
\draw (0, 0)--(210:\s);
\fill (0, 0) circle (2pt);
\node at (75:0.5*\s) {$7$};
\node at (180:0.5*\s) {$4$};
\node at (-75:0.5*\s) {$7$};
\end{tikzpicture}}
\subcaptionbox{\label{abnormalf}}[0.2\linewidth]{\begin{tikzpicture}
\draw [line width = 1.5pt, blue] (0:\s)--(30:\s)--(60:\s)--(90:\s)--(120:\s)--(150:\s)--(180:\s)--(210:\s)--(240:\s)--(270:\s)--(300:\s)--(330:\s)--cycle;
\foreach \ang in {0, 30, 60, 90, 120, 150, 180, 210, 240, 270, 300, 330}{\fill 
(\ang:\s) circle (2pt);}
\draw (0.4*\s, 0)--(30:\s);
\draw (0.4*\s, 0)--(330:\s);
\draw (-0.4*\s, 0)--(150:\s);
\draw (-0.4*\s, 0)--(210:\s);
\draw (-0.4*\s, 0)--(0.4*\s, 0);
\fill (0.4*\s, 0) circle (2pt);
\fill (-0.4*\s, 0) circle (2pt);
\node at (0:0.7*\s) {$4$};
\node at (180:0.7*\s) {$4$};
\node at (90:0.5*\s) {$7$};
\node at (270:0.5*\s) {$7$};
\end{tikzpicture}}
\caption{The abnormal $12^{-}$-cycles in blue.}
\label{abnormal}
\end{figure}
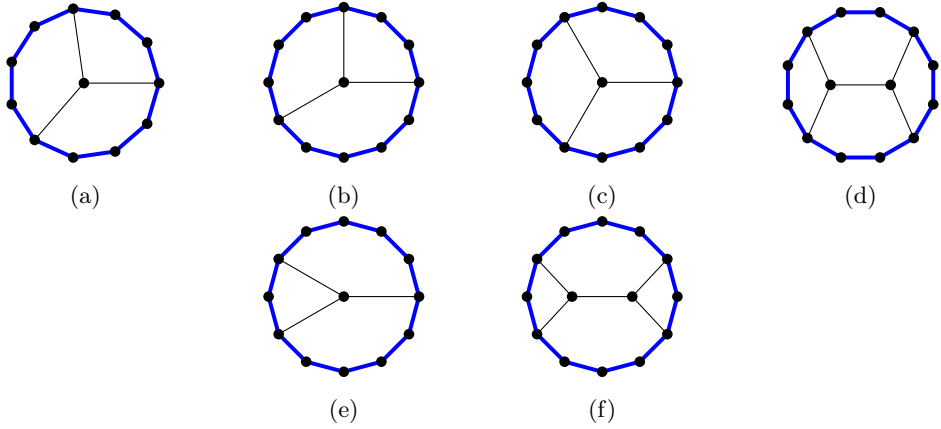

Let $\mathscr{G}$ be the class of plane graphs without triangles normally adjacent to $8^{-}$-cycles, without $4$-cycles normally adjacent to $6^{-}$-cycles, and without normally adjacent $5$-cycles. 

\begin{restatable}{theorem}{MR}\label{MR}
Let $G$ be a graph in $\mathscr{G}$. Let $S$ be a set of vertices of $G$ such that $|S| \leq 1$ or $S$ consists of all vertices on a normal cycle of $G$. Let $\mathscr{M}$ be a $3$-matching assignment for $G$ such that $\mathscr{M}$ is consistent on every closed walk of length three in $G$. If $|S| \leq 12$, then every $\mathscr{M}$-coloring $\phi$ of $G[S]$ can be extended to an $\mathscr{M}$-coloring of $G$. 
\end{restatable}
\begin{remark}
The graphs in \autoref{abnormal} are in the class $\mathscr{G}$. It is observed that not every $\mathscr{M}$-coloring of the $11$- or $12$-cycle can be extended to the whole graph. Thus, we require the condition that $S$ consists of all vertices on a ``normal'' cycle. On the other hand, we find all the non-extendable $12^{-}$-cycles. 
\end{remark}

\begin{enumerate}[label = \textbf{Observation:}, leftmargin = 16ex]
\item Every $10^{-}$-cycle is normal. For a normal $12^{-}$-cycle $\mathcal{O}$, every vertex not on $\mathcal{O}$ has at most two neighbors on $\mathcal{O}$. Every edge on an abnormal $12^{-}$-cycle is contained in a $4$-, $5$-, $6$- or $7$-cycle. 
\end{enumerate}

The following result is a direct consequence of \autoref{MR}, and it extends \autoref{adj8-}. 
\begin{theorem}\label{MR-}
Every graph in $\mathscr{G}$ is $3$-choosable. 
\end{theorem}

The following three results are immediate consequences of \autoref{MR-}. The first one generalizes the $3$-colorability of such graphs described by Luo, Chen and Wang \cite{MR2330169}, and the second one generalizes the $3$-colorability of such graphs described by Wang and Chen \cite{MR2362965}. 
\begin{corollary}
Every planar graph without $4$-, $6$-, $8$-cycles is $3$-choosable. 
\end{corollary}

\begin{corollary}
Every planar graph without $4$-, $5$-, $7$-, $8$-cycles is $3$-choosable. 
\end{corollary}

\begin{remark}
\autoref{3C}, \autoref{adj8-} and \autoref{MR-} are only for $3$-choosable, but not for DP-$3$-colorable. Since we require the ``consistency'' on every closed walk of length three, the current arguments cannot guarantee the DP-3-colorability of the graphs in \autoref{N4-8}, \autoref{NON} and \autoref{MR}. So it is interesting to know whether such graphs are DP-$3$-colorable. 
\end{remark}

Note that we only require the ``consistency'' on every closed walk of length three, if triangles are forbidden in a graph, then we can obtain the following results on DP-3-coloring. 
\begin{corollary}[Liu \etal \cite{MR3886261}]
Every planar graph without $3$-, $5$-, $6$-cycles is DP-3-colorable. 
\end{corollary}

\begin{corollary}[Liu \etal \cite{MR3886261}]
Every planar graph without $3$-, $6$-, $7$-, $8$-cycles is DP-3-colorable. 
\end{corollary}

There are some other sufficient conditions for planar graphs to be DP-$3$-colorable which extend the $3$-choosability of such graphs. We refer the reader to \cite{MR3969021,MR3886261,MR3954054}. DP-$4$-colorable planar or toroidal graphs can be found in \cite{MR4078909,MR4294211,Li2019,MR3996735}. Thomassen \cite{MR1290638} showed that every planar graph is $5$-choosable. Dvo\v{r}\'{a}k and Postle \cite{MR3758240} observed that every planar graph is DP-$5$-colorable. Recently, Li and Wang \cite{MR4401835} extended these results to $K_{5}$-minor-free graphs. 

In addition, we prove a result on the vertex partition using almost the same technique as that in \autoref{MR}. An \textbf{\textit{IF}-coloring} of a graph $G$ is a mapping $\phi: V(G) \rightarrow \{I, F\}$ such that the subgraph induced by all the vertices $\phi^{\mbox{-}1}(I)$ is an independent set, and the subgraph induced by all the vertices $\phi^{\mbox{-}1}(F)$ is a forest. In other words, an \textbf{\textit{IF}-coloring} of a graph $G$ is a partition of $V(G)$ into two parts $I$ and $F$, such that $G[I]$ is an independent set and $G[F]$ is a forest. 

Kawarabayashi and Thomassen \cite{MR2518200} proved that every planar graph with girth at least five has an \textit{IF}-coloring. Wang and Chen \cite{MR2390470} showed that every planar graph without $4$-, $6$- and $8$-cycles is $3$-colorable. Very recently, Liu and Yu \cite{MR4115511} proved that every planar graph without $4$-, $6$- and $8$-cycles has an \textit{IF}-coloring. We prove that every graph in $\mathscr{G}$ has an \textit{IF}-coloring. 

\begin{theorem}\label{IFMR}
The vertex set of every graph in $\mathscr{G}$ can be partitioned into an independent set and a set that induces a forest. 
\end{theorem}

This paper is organized as follows. In the remainder of this section, we introduce some notations and results utilized in the proof of \autoref{MR}. In \autoref{sec:2}, we present a proof of \autoref{MR}. In \autoref{sec:3}, we prove a slightly stronger result than \autoref{IFMR}. Finally, we conclude with some open questions in \autoref{sec:4}. 

Let $G$ be a plane graph. The edges and vertices divide the plane into a number of \textbf{faces}. The unbounded face is called the \textbf{outer face}, and the other faces are called \textbf{inner faces}. An \textbf{internal vertex} is a vertex that is not incident with the outer face. An \textbf{internal face} is a face having no common vertices with the outer cycle. Let $\mathcal{O}$ be a cycle of a plane graph $G$, the cycle $\mathcal{O}$ divides the plane into two regions, the subgraph induced by all the vertices in the unbounded region is denoted by $\ext(\mathcal{O})$, and the subgraph induced by all the vertices in the other region is denoted by $\int(\mathcal{O})$. If both $\int(\mathcal{O})$ and $\ext(\mathcal{O})$ contain at least one vertex, then we call the cycle $\mathcal{O}$ a \textbf{separating cycle} of $G$. The subgraph obtained from $G$ by deleting all the vertices in $\ext(\mathcal{O})$ is denoted by $\Int(\mathcal{O})$, and the subgraph obtained from $G$ by deleting all the vertices in $\int(\mathcal{O})$ is denoted by $\Ext(\mathcal{O})$. Let $\mathcal{N}$ be the set of inner faces having at least one common vertex with the outer face.

We need the following two special covers of cycles. 
\begin{itemize}
\item The \textbf{circular ladder $\Gamma_{n}$} is the Cartesian product of the cycle $C_{n}$ and an independent set with two vertices. 

\item The \textbf{M\"{o}bius ladder} $M_{n}$ is the graph with vertex set $\big\{\,(i, j) \mid i \in [n], j \in [2]\,\big\}$, in which two vertices $(i, j)$ and $(i', j')$ are adjacent if and only if either
\begin{enumerate}[label = ---]
\item $i' = i + 1$ and $j = j'$ for $1 \leq i \leq n-1$, or 
\item $i = n$, $i'=1$ and $j \neq j'$. 
\end{enumerate}
\end{itemize}

The following result can be derived from a theorem in \cite{MR4357325,MR3948125}. 

\begin{theorem}\label{CYCLE-CHOOSABLE}
Let $C$ be a cycle, and let $H$ be a cover with a $2$-list assignment. If the cover is not a circular ladder or a M\"{o}bius ladder, then $H$ has an $\mathscr{M}$-coloring. 
\end{theorem}

\section{Proof of \autoref{MR}}
\label{sec:2}
In this section, we prove the following main result.
\MR*

\begin{proof}
Suppose that $G$ is a minimal counterexample to \autoref{MR}. That is, there exists an $\mathscr{M}$-coloring of $G[S]$ that cannot be extended to an $\mathscr{M}$-coloring of $G$ such that 
\begin{equation}\label{EQ1}
|V(G)| + |E(G)| - |S| \text{ is minimized.}
\end{equation}
Subject to \eqref{EQ1}, 
\begin{equation}\label{EQ2}
\text{the number of edges in the $3$-matching assignment $\mathscr{M}$ is maximized.} 
\end{equation}

By the condition \eqref{EQ2}, each edge that is not in a triangle is full in the matching assignment $\mathscr{M}$. By the structure of $G$, we immediately have the following result on $8^{-}$-cycles. 
\begin{lemma}\label{3-8}
Every $8^{-}$-cycle has no chords. 
\end{lemma}

Next, we give some structural results on $G$. Some of the lemmas are almost the same as the ones in \cite{MR3758240} and \cite{MR3983123}, but for completeness we give detailed proofs here. 
\begin{lemma}\label{LT}
\text{}
\begin{enumerate}[label = (\alph*)]
\item\label{La} $S \neq V(G)$. 
\item\label{Lb} $G$ is $2$-connected, and the boundary of every face is a cycle. 
\item\label{Lc} Each vertex not in $S$ has degree at least three. 
\item\label{Ld} Either $|S| = 1$ or $G[S]$ is an induced cycle of $G$. 
\item\label{Le} There are no separating normal $k$-cycles for $3 \leq k \leq 12$. Thus, every edge on an abnormal cycle is incident with a $4$-, $5$-, $6$- or $7$-face. 
\item\label{Lf} $G[S]$ is an induced cycle of $G$. For convenience, we can redraw the graph $G$ such that $G[S]$ is the outer cycle $\mathcal{C}$ of $G$. Let $D$ be the outer face which is bounded by $G[S]$. 
\item\label{Lg} Every $5$-face ($\neq$ outer face) is incident with at most one $2$-vertex. 
\end{enumerate}
\end{lemma}
\begin{proof}
\ref{La} Suppose to the contrary that $S = V(G)$. Every $\mathscr{M}$-coloring of $G[S]$ is an $\mathscr{M}$-coloring of $G$, a contradiction. 

\ref{Lb} By the condition \eqref{EQ1}, $G$ is connected. Suppose to the contrary that $G$ has a cut-vertex $w$. We may assume that $G = G_{1} \cup G_{2}$ and $G_{1} \cap G_{2} = \{w\}$. By the assumption of the set $S$, either $S \subseteq V(G_{1})$ or $S \subseteq V(G_{2})$. We may assume that $S \subseteq V(G_{1})$. By the condition \eqref{EQ1}, the $\mathscr{M}$-coloring $\phi$ of $G[S]$ can be extended to an $\mathscr{M}$-coloring $\phi_{1}$ of $G_{1}$, and $\phi_{1}(w)$ can be extended to an $\mathscr{M}$-coloring $\phi_{2}$ of $G_{2}$. These two colorings $\phi_{1}$ and $\phi_{2}$ together give an $\mathscr{M}$-coloring of $G$ whose restriction on $G[S]$ is $\phi$, a contradiction. 

\ref{Lc} Suppose that there exists a vertex $w$ not in $S$ having degree two. By the condition \eqref{EQ1}, the $\mathscr{M}$-coloring of $G[S]$ can be extended to an $\mathscr{M}$-coloring of $G - w$. Since $w$ has degree two, there are at most two forbidden colors for $w$, thus we can extend the $\mathscr{M}$-coloring of $G - w$ to an $\mathscr{M}$-coloring of $G$, a contradiction.

\ref{Ld} If $S = \emptyset$, then we put any vertex into $S$ to make $|S| = 1$. Suppose that $S = V(\mathcal{Q})$ and $\mathcal{Q}$ is a cycle with a chord $uv$. It is observed that the $\mathscr{M}$-coloring of $G[S]$ is also an $\mathscr{M}$-coloring of the induced subgraph in $G - uv$. By the condition \eqref{EQ1}, the $\mathscr{M}$-coloring $\phi$ of $G[S]$ can be extended to an $\mathscr{M}$-coloring of $G - uv$, and hence it is also an $\mathscr{M}$-coloring of $G$, a contradiction. 

\ref{Le} We first show that $G[S]$ cannot be a separating cycle. Suppose to the contrary that $G[S]$ is a separating (normal) cycle $\mathcal{O}$. By the condition \eqref{EQ1}, the $\mathscr{M}$-coloring $\phi$ of $\mathcal{O}$ can be extended to an $\mathscr{M}$-coloring $\phi_{1}$ of $\Int(\mathcal{O})$, and another $\mathscr{M}$-coloring $\phi_{2}$ of $\Ext(\mathcal{O})$. These two colorings $\phi_{1}$ and $\phi_{2}$ together give an $\mathscr{M}$-coloring of $G$ whose restriction on $G[S]$ is $\phi$, a contradiction. 

Thus, either $|S| = 1$ or $S$ consists of all vertices on a face of $G$. Let $\mathcal{Q}$ be a separating normal $k$-cycle with $3 \leq k \leq 12$. Thus, we may assume that $S \subseteq \Ext(\mathcal{Q})$. By the condition \eqref{EQ1}, the $\mathscr{M}$-coloring $\phi$ of $G[S]$ can be extended to an $\mathscr{M}$-coloring $\varphi_{1}$ of $\Ext(\mathcal{Q})$. Similarly, the restriction of $\varphi_{1}$ on $Q$ can be extended to an $\mathscr{M}$-coloring $\varphi_{2}$ of $\Int(\mathcal{Q})$. These two colorings $\varphi_{1}$ and $\varphi_{2}$ together give an $\mathscr{M}$-coloring of $G$ whose restriction on $G[S]$ is $\phi$, a contradiction. 

\ref{Lf} According to \ref{Ld}, suppose to the contrary that $S = \{w\}$. We first assume that $w$ is on a $10^{-}$-cycle $\mathcal{Q}$. Without loss of generality, we may assume that $\mathcal{Q}$ is a shortest cycle containing $w$. Then $\mathcal{Q}$ is an induced cycle. By \ref{Le}, we may assume that $\ext(\mathcal{Q}) = \emptyset$ and $\mathcal{Q}$ is the outer cycle. By \ref{Lc} and $\mathcal{Q}$ is an induced cycle, every vertex on $\mathcal{Q}$ other than $w$ has a neighbor in $\int(\mathcal{Q})$, which implies that $\int(\mathcal{Q}) \neq \emptyset$. By the condition \eqref{EQ1}, the $\mathscr{M}$-coloring $\phi$ of $\{w\}$ can be extended to an $\mathscr{M}$-coloring $\phi_{1}$ of $\mathcal{Q}$. By the condition \eqref{EQ1}, the $\mathscr{M}$-coloring $\phi_{1}$ of $\mathcal{Q}$ can be further extended to an $\mathscr{M}$-coloring of $G$, a contradiction. 

So we may assume that every cycle containing $w$ has length at least $11$. Let $w$ be incident with a face $[w_{1}ww_{2}\dots w_{1}]$. Let $G'$ be obtained from $G$ by adding a chord $w_{1}w_{2}$ in the face, let $S' = \{w, w_{1}, w_{2}\}$ and let the $3$-matching assignment $\mathscr{M}'$ for $G'$ be obtained from $\mathscr{M}$ by setting the matching corresponding to $w_{1}w_{2}$ is edgeless. We can easily check that $G'$ is a plane graph satisfying the assumption of \autoref{MR}. By the condition \eqref{EQ1}, the $\mathscr{M}$-coloring $\phi$ of $\{w\}$ can be extended to an $\mathscr{M'}$-coloring $\phi_{1}$ of $G'[S']$. By the condition \eqref{EQ1}, the $\mathscr{M'}$-coloring $\phi_{1}$ of $G'[S']$ can be further extended to an $\mathscr{M'}$-coloring $\varphi$ of $G'$. It is observed that $\varphi$ is an $\mathscr{M}$-coloring of $G$, a contradiction. 

\ref{Lg} Note that every $2$-vertex and its two neighbors are all on the outer cycle. Suppose to the contrary that $f = [x_{1}x_{2}x_{3}x_{4}x_{5}]$ is a $5$-face which is incident with two $2$-vertices. Note that the two $2$-vertices must be adjacent on the outer cycle, say $x_{2}$ and $x_{3}$. It follows that $x_{1}$ and $x_{4}$ are on the outer cycle $\mathcal{C}$. If $x_{5}$ has three neighbors on the outer cycle $\mathcal{C}$, then $\mathcal{C}$ is abnormal (see \autoref{abnormala} and \autoref{abnormalb}), a contradiction. Thus, $x_{5}$ has a neighbor not on the outer cycle $\mathcal{C}$, and $\mathcal{C}'= (\mathcal{C} - \{x_{2}, x_{3}\}) \cup \{x_{1}x_{5}, x_{4}x_{5}\}$ is a separating $11^{-}$-cycle. By \ref{Le}, $\mathcal{C}'$ is an abnormal $11$-cycle (see \autoref{abnormala}). It follows that $\mathcal{C}$ is an abnormal $12$-cycle (see \autoref{abnormald}), a contradiction. 
\end{proof} 

\begin{lemma}\label{ADJF}
There are no $3$-faces adjacent to $8^{-}$-faces, no $4$-faces adjacent to $6^{-}$-faces, and no adjacent $5$-faces. 
\end{lemma}
\begin{proof}
Recall that every face is bounded by a cycle. Assume that $f$ is an $8^{-}$-face and it is adjacent to a $5^{-}$-face $g$. By \autoref{3-8}, it suffices to consider that $g$ is a $4$- or $5$-face. 

Suppose that $f = [w_{1}w_{2}\dots w_{k}]$ is a $6^{-}$-face and $g = [uvw_{3}w_{2}]$ is a $4$-face. Since there are no $4$-cycles normally adjacent to $6^{-}$-cycles, we have that either $u$ or $v$ is on $f$. By symmetry, we may assume that $u$ is on $f$. Recall that every $6^{-}$-face is bounded by a cycle and this cycle has no chord, so $u = w_{1}$ and $v$ is not on $f$. It is observed that $w_{2}$ is a $2$-vertex and it must be on the outer cycle $\mathcal{C}$. It follows that either $f$ or $g$ is the outer face. If $f$ is the outer face, then $v$ is an internal vertex and it has a neighbor not on the outer cycle $\mathcal{C}$ (since $[w_{1}vw_{3}\dots w_{k}]$ has no chords by \autoref{3-8}), thus there is a separating $6^{-}$-cycle $[w_{1}vw_{3}\dots w_{k}]$, a contradiction. Similarly, if $g$ is the outer face, then there is an internal vertex on $f$ having a neighbor not on the outer cycle $\mathcal{C}$, thus there is a separating $6^{-}$-cycle containing $w_{1}vw_{3}$, a contradiction. 

Suppose that $f$ and $g$ are $5$-faces. Since every $8^{-}$-cycle has no chord, there are only four cases (up to symmetry) for the local structures (see \autoref{5-8}). Since there are no normally adjacent $5$-cycles, the first case will not appear. For the other three cases, we first assume that $x$ is an internal vertex. Since every internal vertex has degree at least three, $x$ has a neighbor $x'$ other than $w_{2}$ and $y$. It is observed that $x$ is on a $6^{-}$-cycle $\mathcal{O}_{x}$ not containing $w_{3}$. If $x'$ is on $\mathcal{O}_{x}$, then $xx'$ is a chord of $\mathcal{O}_{x}$, but this contradicts \autoref{3-8}; if $x'$ is not on $\mathcal{O}_{x}$, then $\mathcal{O}_{x}$ is a separating $6^{-}$-cycle, this contradicts \autoref{LT}\ref{Le}. So we may assume that $x$ is on the outer cycle. In the second and third cases, $w_{3}$ is a $2$-vertex, so it is on the outer cycle, and $g$ must be the outer face. In the fourth case, by the symmetry of $x$ and $z$, we have that $z$ is on the outer cycle, and $g$ is the outer face. Therefore, $g$ is the outer face in the last three cases, and there is an internal vertex on $f$ having a neighbor not on $g$, and then there is a separating $6^{-}$-cycle containing $w_{2}xy$, a contradiction. 
\end{proof}

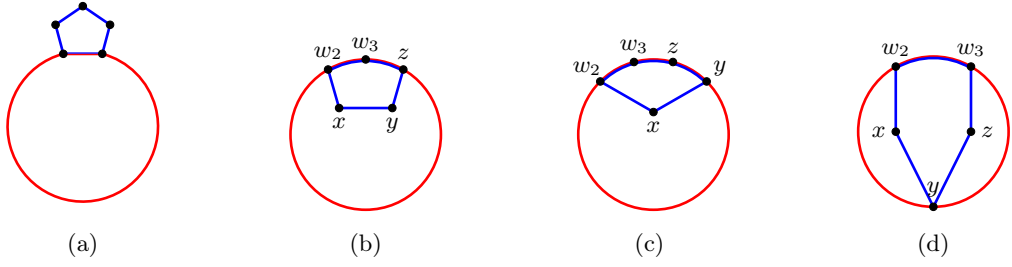
\begin{figure}%
\def\s{1}
\centering
\subcaptionbox{}{\begin{tikzpicture}
\coordinate (u) at (105:1.4*\s);
\coordinate (w2) at (105:1*\s);
\coordinate (w3) at (75:1*\s);
\coordinate (w) at (75:1.4*\s);
\coordinate (v) at (0, 1.6*\s);
\draw[line width=1pt, blue] (w2)--(u)--(v)--(w)--(w3);
\draw[line width=0.55pt, blue] (105:1.01*\s)--(75:1.01*\s);
\draw[line width=1pt, red] (105:1*\s) arc (105:75+360:1*\s);
\draw[line width=0.55pt, red] (105:0.99*\s)--(75:0.99*\s);
\node[circle, inner sep =1, fill, draw] () at (u) {};
\node[circle, inner sep =1, fill, draw] () at (v) {};
\node[circle, inner sep =1, fill, draw] () at (w) {};
\node[circle, inner sep =1, fill, draw] () at (w2) {};
\node[circle, inner sep =1, fill, draw] () at (w3) {};
\end{tikzpicture}}\hspace{1.5cm}
\subcaptionbox{}{\begin{tikzpicture}
\draw[line width=1pt, red] (0,0) circle (1*\s); 
\coordinate (u) at (120:1*\s);
\coordinate (w2) at (90:1*\s);
\coordinate (w3) at (60:1*\s);
\coordinate (w) at (45:0.5*\s);
\coordinate (v) at (135:0.5*\s);
\draw[line width=1pt, blue] (u) node [above, black]{\small $w_{2}$}--(v) node [below, black]{\small $x$}--(w) node [below, black]{\small $y$}--(w3) node [above, black]{\small $z$};
\draw[line width=1pt, blue] (60:0.98*\s) arc (60:90:0.98*\s) node [above, black]{\small $w_{3}$};
\draw[line width=1pt, blue] (90:0.98*\s) arc (90:120:0.98*\s);
\node[circle, inner sep =1, fill, draw] () at (u) {};
\node[circle, inner sep =1, fill, draw] () at (v) {};
\node[circle, inner sep =1, fill, draw] () at (w) {};
\node[circle, inner sep =1, fill, draw] () at (w2) {};
\node[circle, inner sep =1, fill, draw] () at (w3) {};
\end{tikzpicture}}\hspace{1.5cm}
\subcaptionbox{}{\begin{tikzpicture}
\draw[line width=1pt, red] (0,0) circle (1*\s); 
\coordinate (u) at (135:1*\s);
\coordinate (w2) at (105:1*\s);
\coordinate (w3) at (75:1*\s);
\coordinate (w) at (45:1*\s);
\coordinate (v) at (0, 0.3*\s);
\draw[line width=1pt, blue] (u)--(v) node [below, black]{\small $x$}--(w);
\draw[line width=1pt, blue] (45:0.98*\s) arc (45:135:0.98*\s);
\node[circle, inner sep =1, fill, draw] () at (u) {};
\node at (135:1.25*\s) {\small $w_{2}$};
\node[circle, inner sep =1, fill, draw] () at (v) {};
\node[circle, inner sep =1, fill, draw] () at (w) {};
\node at (45:1.25*\s) {\small $y$};
\node[circle, inner sep =1, fill, draw] () at (w2) {};
\node[above] at (w2) {\small $w_{3}$};
\node[circle, inner sep =1, fill, draw] () at (w3) {};
\node[above] at (w3) {\small $z$};
\end{tikzpicture}}\hspace{1.5cm}
\subcaptionbox{}{\begin{tikzpicture}
\draw[line width=1pt, red] (0,0) circle (1*\s); 
\coordinate (w2) at (120:1*\s);
\coordinate (w3) at (60:1*\s);
\coordinate (u) at (180:0.5*\s);
\coordinate (v) at (-90:1*\s);
\coordinate (w) at (0:0.5*\s);
\draw[line width=1pt, blue] (w2) node [above, black]{\small $w_{2}$}--(u) node [left, black]{\small $x$} --(v) node [above, black]{\small $y$}--(w) node [right, black]{\small $z$}--(w3) node [above, black]{\small $w_{3}$};
\draw[line width=1pt, blue] (60:0.98*\s) arc (60:120:0.98*\s);
\node[circle, inner sep =1, fill, draw] () at (u) {};
\node[circle, inner sep =1, fill, draw] () at (v) {};
\node[circle, inner sep =1, fill, draw] () at (w) {};
\node[circle, inner sep =1, fill, draw] () at (w2) {};
\node[circle, inner sep =1, fill, draw] () at (w3) {};
\end{tikzpicture}}
\caption{A $5$-face is adjacent to an $8^{-}$-face, where the blue cycle bounds a $5$-face and the red cycle bounds an $8^{-}$-face.}
\label{5-8}
\end{figure}

\begin{lemma}\label{full}
If $[x_{1}x_{2}x_{3}]$ is a triangle and $x_{2}, x_{3}$ are internal $3$-vertices, then all the edges in the triangle are full. 
\end{lemma}
\begin{proof}
Suppose to the contrary that at least one of $x_{1}x_{3}, x_{2}x_{3}$ and $x_{1}x_{2}$ is not full. By applying \autoref{ST} to $\{x_{1}x_{3}, x_{2}x_{3}\}$, we may assume that $x_{1}x_{3}$ and $x_{2}x_{3}$ are straight in $\mathscr{M}$. Let $\mathscr{M}'$ be a new $3$-matching assignment for $G$ by setting $\mathscr{M}_{e}' = \mathscr{M}_{e}$ for each $e \notin \{x_{1}x_{3}, x_{2}x_{3}, x_{1}x_{2}\}$ and all edges in $\{x_{1}x_{3}, x_{2}x_{3}, x_{1}x_{2}\}$ are straight and full. Note that $x_{1}x_{3}$ and $x_{2}x_{3}$ are straight in $\mathscr{M}$, thus $\mathscr{M}_{x_{1}x_{3}} \subseteq \mathscr{M}_{x_{1}x_{3}}'$ and $\mathscr{M}_{x_{2}x_{3}} \subseteq \mathscr{M}_{x_{2}x_{3}}'$. Since all the edges in $\{x_{1}x_{3}, x_{2}x_{3}, x_{1}x_{2}\}$ are full in $\mathscr{M}'$ but not in $\mathscr{M}$, the number of edges in $\mathscr{M}'$ is greater than that in $\mathscr{M}$. Since there are no adjacent triangles, every closed walk of length three is consistent in $\mathscr{M}'$. By the condition \eqref{EQ2}, the $\mathscr{M}$-coloring $\phi$ (also $\mathscr{M}'$-coloring) of the outer cycle $\mathcal{C}$ can be extended to an $\mathscr{M}'$-coloring $\phi'$ of $G$, but $\phi'$ is not an $\mathscr{M}$-coloring of $G$ by our assumption. Note that $\mathscr{M}_{e} \subseteq \mathscr{M}_{e}'$ for any $e \neq x_{1}x_{2}$, so we may assume that $\phi'(x_{1}) = 1$, $\phi'(x_{2}) = 2$ and $(x_{1}, 1)(x_{2}, 2) \in \mathscr{M}_{x_{1}x_{2}}$. If $(x_{1}, 1)$ has an incident edge in $\mathscr{M}_{x_{1}x_{3}}$ and $(x_{2}, 2)$ has an incident edge in $\mathscr{M}_{x_{2}x_{3}}$, then the closed walk $x_{3}x_{1}x_{2}$ is not consistent in $\mathscr{M}$, a contradiction. If $(x_{1}, 1)$ has no incident edge in $\mathscr{M}_{x_{1}x_{3}}$, then we can modify $\phi'$ to obtain an $\mathscr{M}$-coloring of $G$ by recoloring $x_{2}$ and $x_{3}$ in order, a contradiction. So we may assume that $(x_{1}, 1)$ has an incident edge in $\mathscr{M}_{x_{1}x_{3}}$ and $(x_{2}, 2)$ has no incident edge in $\mathscr{M}_{x_{2}x_{3}}$. Since $x_{1}x_{3}$ is straight in $\mathscr{M}$, we have that $(x_{1}, 1)(x_{3}, 1) \in \mathscr{M}_{x_{1}x_{3}}$. Furthermore, we may assume that $(x_{3}, 1)$ has no incident edge in $\mathscr{M}_{x_{2}x_{3}}$, otherwise the closed walk $x_{2}x_{3}x_{1}$ is not consistent in $\mathscr{M}$. Now, we can obtain a new $3$-matching assignment $\mathscr{M}^{*}$ for $G$ by adding an edge $(x_{3}, 1)(x_{2}, 2)$ to $\mathscr{M}$. By the adjacency of cycles, $x_{2}x_{3}$ is only contained in an unique triangle $x_{1}x_{2}x_{3}$, so the addition of $(x_{3}, 1)(x_{2}, 2)$ does not make $\mathscr{M}^{*}$ inconsistent on closed $3$-walk, but this contradicts the condition \eqref{EQ2}. 
\end{proof}

\begin{figure}
\centering
\begin{tikzpicture}
\def\s{5}
\coordinate (w0) at (127.5:\s);
\coordinate (w1) at (112.5:\s);
\coordinate (w2) at (97.5:\s);
\coordinate (w3) at (82.5:\s);
\coordinate (w4) at (67.5:\s);
\coordinate (w) at (105:1.2*\s);
\coordinate (w') at (75:1.2*\s);
\draw (45:\s) arc(45:135:\s);
\draw (w1)--(w)--(w2);
\draw (w3)--(w');
\draw[dashed] (w0)..controls+(90:1cm) and +(180:1cm)..(w)--(w');
\node[circle, inner sep =1, fill, draw] () at (w0) {};
\node[circle, inner sep =1, fill, draw] () at (w1) {};
\node[circle, inner sep =1, fill, draw] () at (w2) {};
\node[circle, inner sep =1, fill, draw] () at (w3) {};
\node[circle, inner sep =1, fill, draw] () at (w4) {};
\node[circle, inner sep =1, fill, draw] () at (w) {};
\node[circle, inner sep =1, fill, draw] () at (w') {};
\node[below] at (w0) {$w_{0}$};
\node[below] at (w1) {$w_{1}$};
\node[below] at (w2) {$w_{2}$};
\node[below] at (w3) {$w_{3}$};
\node[below] at (w4) {$w_{4}$};
\node[above] at (w) {$w$};
\node[above] at (w') {$w'$};
\end{tikzpicture}
\caption{A case in \autoref{SIMITETRAD}.}
\label{WON}
\end{figure}

\begin{lemma}\label{SIMITETRAD}
Let $w_{0}, w_{1}, w_{2}, w_{3}$ and $w_{4}$ be five consecutive vertices on a $5^{+}$-face. If $w_{1}, w_{2}, w_{3}$ and $w_{4}$ are all $3$-vertices, and $w_{1}w_{2}$ is incident with a $3$-face $ww_{1}w_{2}$, then at least one vertex in $\{w_{1}, w_{2}, w_{3}, w_{4}\}$ is on the outer cycle $\mathcal{C}$. 
\end{lemma}
\begin{proof}
Suppose to the contrary that none of $\{w_{1}, w_{2}, w_{3}, w_{4}\}$ is on the outer cycle $\mathcal{C}$. Let $w'$ be the neighbor of $w_{3}$ other than $w_{2}, w_{4}$, and let $G^{*} = G - \{w_{1}, w_{2}, w_{3}, w_{4}\}$. It is observed that $w_{0}, w_{1}, w_{2}, w_{3}, w_{4}, w$ and $w'$ are seven distinct vertices. We claim that the distance between $w_{0}$ and $w'$ is at least nine in $G^{*}$. Let $P$ be a shortest path between $w_{0}$ and $w'$ in $G^{*}$. It is observed that $\mathcal{Q} = P \cup w_{0}w_{1}w_{2}w_{3}w'$ is a cycle. If $w$ is on the path $P$, then $P[w_{0}, w] \cup w_{0}w_{1}w$ and $P[w, w'] \cup ww_{2}w_{3}w'$ are all cycles (see \autoref{WON}), which implies that these two cycles have length at least nine and $|E(P)| \geq (9-2) + (9-3) = 13$. If $w$ is not on the path $P$, then $\mathcal{Q}$ is a separating normal cycle (note that $w$ and $w_{4}$ are in different sides of the cycle $\mathcal{Q}$) and it has length at least $13$, which implies that $|E(P)| = |\mathcal{Q}| - 4 \geq 13 - 4 = 9$. Therefore, the distance between $w_{0}$ and $w'$ is at least nine in $G^{*}$. 

By \autoref{full} and \autoref{ST}, we may assume that all the edges incident with the vertices in $\{w_{1}, w_{2}, w_{3}\}$ are straight. Let $G'$ be the graph obtained from $G^{*}$ by identifying $w_{0}$ and $w'$, and let $\mathscr{M'}$ be the restriction of $\mathscr{M}$ on $E(G')$. Since the distance between $w_{0}$ and $w'$ is at least nine in $G^{*}$, the graph $G'$ has no loops, no multiple edges and no new $8^{-}$-cycles, thus $G'$ is a simple plane graph satisfying the assumption of \autoref{MR}, and $\mathscr{M}'$ is consistent on every closed walk of length three in $G'$. Moreover, $\mathcal{C}$ is also a normal cycle of $G'$ and it has no chords in $G'$. This implies that $\phi$ is an $\mathscr{M}'$-coloring of $G'[S]$. Since $|V(G')| < |V(G)|$, the $\mathscr{M}'$-coloring $\phi$ of $G'[S]$ can be extended to an $\mathscr{M}'$-coloring $\varphi$ of $G'$. Since $w_{3}$ and $w_{4}$ are all $3$-vertices, we can extend $\varphi$ to $w_{4}$ and $w_{3}$ in order. Recall that all the edges incident with vertices in $\{w_{1}, w_{2}, w_{3}\}$ are straight, thus we may assume that $w_{0}$ and $w_{3}$ have distinct colors, and then we can further extend the coloring to $w_{2}$ and $w_{1}$, a contradiction. 
\end{proof}

Let $w$ be a vertex on the outer cycle $\mathcal{C}$, and let $w_{1}, w_{2}, \dots, w_{k}$ be consecutive neighbors in a cyclic order. If $f$ is a face in $\mathcal{N}$ incident with $ww_{i}$ and $ww_{i+1}$, but neither $ww_{i}$ nor $ww_{i+1}$ is an edge of $\mathcal{C}$, then we call $f$ a \textbf{special} face (at $w$). A $4$-face is a \textbf{$4^{*}$-face} if it has three common vertices with $\mathcal{C}$. An internal $3$-vertex is \textbf{bad} if it is incident with a $3$-face which is not special, \textbf{light} if it is incident with an internal $4$-face or a $4^{*}$-face or a special $3$-face, \textbf{good} if it is neither bad nor light. According to \autoref{SIMITETRAD}, we have the following result on bad vertices. 

\begin{lemma}\label{BAD}
There are no five consecutive bad vertices on the boundary of a $5^{+}$-face. 
\end{lemma}

\begin{lemma}\label{4F}
If a $4$-face in $\mathcal{N}$ has exactly two common vertices with $\mathcal{C}$, then these two vertices are consecutive on the $4$-face. 
\end{lemma}
\begin{proof}
Suppose that $f$ is a $4$-face in $\mathcal{N}$ that has exactly two common vertices with $\mathcal{C}$. If these two vertices are not consecutive on the $4$-face, then there exists a separating normal $8^{-}$-cycle, a contradiction. 
\end{proof}

Assume that $f = [v_{1}v_{2}\dots v_{l}]$ is an internal $(3, 3, 3, 3)$-face or an internal $(3, 3, 3, 3, 3)$-face. Let $u_{i}$ be the third neighbor of $v_{i}$ for $1 \leq i \leq l$. Note that every $8^{-}$-cycle has no chords, and there are no separating $4$- or $5$-cycles. It is observed that $\{u_{i}, v_{i}\mid 1 \leq i \leq l\}$ contains $2l$ distinct vertices. Let $\Gamma$ be the graph $G \setminus (V(f) \setminus \{v_{3}\})$, and let $G^{*}$ be the graph obtained from $\Gamma$ by identifying $u_{1}$ and $v_{3}$ into a new vertex $z$. 

\begin{lemma}\label{CLOSED}
The graph $G^{*}$ is in the class $\mathscr{G}$, and $\mathcal{C}$ is an induced cycle of $G^{*}$. \end{lemma}
\begin{proof}
Let $P$ be an arbitrary path between $u_{1}$ and $v_{3}$ in $\Gamma$. It is observed that $\mathcal{Q} = P \cup u_{1}v_{1}v_{2}v_{3}$ is a cycle of length $|E(P)| + 3$. 

(i) Assume that $u_{2}$ is not on the path $P$. It is easy to check that $|E(P)| \leq 8$ only if $|E(P)| = 8$ and $\mathcal{Q}$ is a separating abnormal $11$-cycle (note that $u_{2}$ and $v_{l}$ are in different sides of the cycle $\mathcal{Q}$). Note that there are no $3$-cycles normally adjacent to abnormal cycles. Even though the identification may make the path $P$ becomes an $8$-cycle, but the new cycle is not normally adjacent to any $3$-cycle in $G^{*}$. 

(ii) Assume that $u_{2}$ is on the path $P$. Then $Q_{1} = P[u_{1}, u_{2}] \cup u_{1}v_{1}v_{2}u_{2}$ and $Q_{2} = P[u_{2}, v_{3}] \cup u_{2}v_{2}v_{3}$ are all cycles, which implies that each of these two cycles has length at least six and $|E(P)| \geq (6 - 2) + (6 - 3) = 7$. It is easy to check that $|E(P)| = 7$ only if $|Q_{1}| = |Q_{2}| = 6$, and $|E(P)| = 8$ only if $|Q_{i}| = 6$ and $|Q_{3-i}| = 7$. Recall that there are no $3$-cycles normally adjacent to $8^{-}$-cycles, and there are no $4$-cycles normally adjacent to $6^{-}$-cycles. Therefore, even though the identification may make the path $P$ become a $7$- or $8$-cycle, but the new cycle is not normally adjacent to any $3$-cycle in $G^{*}$. 

According to the above arguments on the two cases, $G^{*}$ satisfies the requirement for adjacency in \autoref{MR}, and the distance of $u_{1}$ and $v_{3}$ in $\Gamma$ is at least seven. Note that $v_{3}$ is an internal vertex, $\mathcal{C}$ is still a $12^{-}$-cycle in $G^{*}$. It is easy to check that $\mathcal{C}$ is also a normal $12^{-}$-cycle in $G^{*}$, for otherwise, two new $7$-cycles are created to make $\mathcal{C}$ abnormal in $G^{*}$, but there is a separating normal $10$-cycle in $G$, a contradiction. 

Suppose that the identification creates a chord for $\mathcal{C}$. Since $v_{3}$ is an internal vertex, the vertex $u_{1}$ is on the outer cycle $\mathcal{C}$. Note that $u_{3}$ is also on the outer cycle, the two vertices $u_{1}$ and $u_{3}$ divide the cycle $\mathcal{C}$ into two paths $P'$ and $P''$. Since the distance of $u_{1}$ and $v_{3}$ in $\Gamma$ is at least seven, we have that $|E(P')| = |E(P'')| = 6$. Thus, $P' \cup u_{1}v_{1}v_{2}v_{3}u_{3}$ or $P'' \cup u_{1}v_{1}v_{2}v_{3}u_{3}$ is a separating normal $10$-cycle in $G$, a contradiction. Hence, the identification does not create chords of $\mathcal{C}$. 
\end{proof}

\begin{lemma}\label{NO33333}
There are no internal $(3, 3, 3, 3)$-faces or $(3, 3, 3, 3, 3)$-faces. 
\end{lemma}
\begin{proof}
Suppose to the contrary that $f = [v_{1}v_{2}\dots v_{l}]$ is an internal $(3, 3, 3, 3)$-face or an internal $(3, 3, 3, 3, 3)$-face. All the related vertices are labeled as the above. Noting that none of $u_{1}v_{1}, v_{1}v_{2}, v_{2}v_{3}, v_{3}u_{3}$ is in a triangle, it follows that these four edges are full. So we may assume that each of these four edges is straight. Let $H^{*}$ be the cover obtained from $H - \cup L_{v}$ by identifying $(u_{1}, j)$ and $(v_{3}, j)$ into a new vertex $(z, j)$, where $1 \leq j \leq 3$, and the union is taken over $v \in V(f) \setminus \{v_{3}\}$. By the minimality of $G$, the precoloring $\phi$ can be extended to a coloring $T^{*}$ of the cover $H^{*}$. Without loss of generality, we may assume that $(z, 1) \in T$. This coloring naturally gives a coloring $T$ of $H - \cup_{x \in V(f)} L_{x}$. For each $x \in V(f)$, let $L'_{x} = L_{x} \setminus \{(x, j_{x}) \mid \text{$(x, j_{x})$ is the neighbor of a vertex in $T$}\}$. Let $H'$ be the subgraph of $H$ induced by $\cup_{x \in V(f)} L'_{x}$. Note that $(v_{1}, 1) \notin L'_{v_{1}}$ and $(v_{3}, 1) \in L'_{v_{3}}$. Recall that $v_{1}v_{2}$ and $v_{2}v_{3}$ are full and straight, thus $H'$ is neither a circular ladder nor a M\"{o}bius ladder. By \autoref{CYCLE-CHOOSABLE}, the cover $H'$ has a coloring $T'$. Therefore, $T \cup T'$ is a coloring of the cover $H$, a contradiction. 
\end{proof}

We give the initial charge $\mu(v) = \deg(v) - 4$ for any $v \in V(G)$, $\mu(f) = \deg(f) - 4$ for any face $f \in F(G)$ other than outer face $D$, and $\mu(D) = \deg(D) + 4$ for the outer face $D$. By the Euler formula, the sum of the initial charges is zero. That is, 
\begin{equation}\label{Euler2}
\sum_{v\,\in\,V(G)}\big(\deg(v) - 4\big) + \sum_{f\,\in\,F(G) \setminus D}\big(\deg(f) - 4\big) + (\deg(D) + 4)= 0. 
\end{equation}
Next, we give the discharging rules to redistribute the charges, preserving the sum, such that the final charge of every element in $V(G) \cup F(G)$ is nonnegative, and at least one element in $V(G) \cup F(G)$ has positive final charge. This leads a contradiction to complete the proof. 

A \textbf{small face} is a $5^{-}$-face. We say that a face is a \textbf{negative face} if it is a non-special $3$-face, or an internal $(3, 3, 3, 3, 4)$-face, or an internal $4$-face. 

\medskip
The followings are the discharging rules. 
\begin{enumerate}[label = \textbf{R\arabic*.}, ref = R\arabic*]
\item\label{3F} Each non-special $3$-face receives $\frac{1}{3}$ from each incident internal vertex. 
\item\label{IN3V} Let $w$ be an internal $3$-vertex. 
\begin{enumerate}
\item\label{IN3V1} If $w$ is a bad vertex, then it receives $\frac{2}{3}$ from each incident $9^{+}$-face.
\item\label{IN3V2} If $w$ is incident with a special $3$-face or a $4^{*}$-face, then it receives $\frac{1}{2}$ from each incident $7^{+}$-face. 
\item\label{IN3V3} If $w$ is incident with an internal $4$-face, then it receives $\frac{1}{9}$ from the incident $4$-face, and $\frac{4}{9}$ from each incident $7^{+}$-face. 
\item\label{IN3V4} If $w$ is a good vertex, then it receives $\frac{1}{3}$ from each incident face. 
\end{enumerate}
\item\label{IN4V} Let $w$ be an internal $4$-vertex. Then $w$ sends $\frac{1}{3}$ to each incident internal $(3, 3, 3, 3, 4)$-face and internal $4$-face. 
\begin{enumerate}
\item\label{IN4V1} If $w$ is incident with two negative faces, then it receives $\frac{1}{3}$ from each of the other two incident faces. 
\item\label{IN4V2} If $w$ is incident with exactly one negative face $f$, and the opposite face at $w$ is a $9^{+}$-face, then $w$ receives $\frac{1}{3}$ from the opposite face. 
\item\label{IN4V3} If $w$ is incident with exactly one negative face $f$, and the opposite face at $w$ is an $8^{-}$-face, then $w$ receives $\frac{1}{6}$ from each face at $w$ that is not the opposite face. 
\end{enumerate} 
\item\label{IN5PV} Each internal $5^{+}$-vertex sends $\frac{1}{3}$ to each incident $5$-face. 
\item\label{2V} Each $2$-vertex on the outer cycle $\mathcal{C}$ receives $\frac{2}{3}$ from the incident face in $\mathcal{N}$ and $\frac{4}{3}$ from the outer face.
\item\label{3V} Each $3$-vertex on the outer cycle $\mathcal{C}$ receives $\frac{4}{3}$ from the outer face, and sends $\frac{1}{3}$ to each incident $4^{-}$-face in $\mathcal{N}$ and $\frac{1}{6}$ to each incident $5$- or $6$-face in $\mathcal{N}$.
\item\label{4V} Each $4$-vertex on the outer cycle $\mathcal{C}$ receives $1$ from the outer face, and sends $1$ to each incident special $4^{-}$-face, and $\frac{1}{3}$ to each of the other incident $6^{-}$-face in $\mathcal{N}$. 
\item\label{5PV} Each $5^{+}$-vertex on the outer cycle $\mathcal{C}$ receives $1$ from the outer face, and sends $1$ to each incident special $4^{-}$-face, and $\frac{1}{2}$ to each of the other incident $6^{-}$-face in $\mathcal{N}$. 
\end{enumerate}

\begin{lemma}
Every face other than $D$ has nonnegative final charge. 
\end{lemma}
\begin{proof}
According to the discharging rules, inner $3$-faces never give charges. If $f$ is a special $3$-face, then $\mu'(f) = 3 - 4 + 1 = 0$ by \ref{4V} and \ref{5PV}. If $f$ is a non-special $3$-face having no vertex on the outer cycle $\mathcal{C}$, then $\mu'(f) = 3 - 4 + 3 \times \frac{1}{3} = 0$ by \ref{3F}. If $f$ is a non-special $3$-face having two vertices on the outer cycle $\mathcal{C}$, then $f$ has a common edge with the outer face by \autoref{LT}\ref{Lf}, and then $\mu'(f) \geq 3 - 4 + 3 \times \frac{1}{3} = 0$ by \ref{3F}, \ref{3V}, \ref{4V} and \ref{5PV}. Note that no inner $3$-faces have three common vertices with $\mathcal{C}$. 

Let $f$ be a $4$-face. If $f$ is an internal face, then it is incident with at least one $4^{+}$-vertex by \autoref{NO33333}, and then $\mu'(f) \geq 4 - 4 + \frac{1}{3} - 3 \times \frac{1}{9} = 0$ by \ref{IN3V3} and \ref{IN4V}. So we may assume that $f$ is a face in $\mathcal{N}$. If $f$ has exactly one common vertex with $\mathcal{C}$, then it receives $1$ from the vertex on the outer cycle $\mathcal{C}$, and then $\mu'(f) \geq 4 - 4 + 1 - 3 \times \frac{1}{3} = 0$ by \ref{IN3V4}, \ref{4V} and \ref{5PV}. If $f$ has exactly two common vertices with the outer cycle $\mathcal{C}$, then these two vertices are consecutive on $f$ by \autoref{4F}, and then $\mu'(f) \geq 4 - 4 + 2\times \frac{1}{3} - 2 \times \frac{1}{3} = 0$ by \ref{IN3V4}, \ref{3V}, \ref{4V} and \ref{5PV}. If $f$ has exactly three common vertices with the outer cycle $\mathcal{C}$, \ie $f$ is a $4^{*}$-face, then $\mu'(f) \geq 4 - 4 + 2 \times \frac{1}{3} - \frac{2}{3} = 0$ by \ref{2V}, \ref{4V} and \ref{5PV}. 

Let $f$ be a $5$-face. By \autoref{NO33333}, $f$ cannot be incident with five internal $3$-vertices. If $f$ is an internal $(3, 3, 3, 3, 4^{+})$-face, then $f$ receives $\frac{1}{3}$ from the incident $4^{+}$-vertex and sends $\frac{1}{3}$ to each incident $3$-vertex, thus $\mu'(f) = 5 - 4 + \frac{1}{3} - 4 \times \frac{1}{3} = 0$ by \ref{IN3V4}, \ref{IN4V} and \ref{IN5PV}. If $f$ is an internal face incident with at least two $4^{+}$-vertices, then $\mu'(f) \geq 5 - 4 - 3 \times \frac{1}{3} = 0$ by \ref{IN3V4}. So we may assume that $f$ is a face in $\mathcal{N}$. If $f$ has exactly one common vertex with the outer cycle $\mathcal{C}$, then it receives at least $\frac{1}{3}$ from the vertex on the outer cycle $\mathcal{C}$, and then $\mu'(f) \geq 5 - 4 + \frac{1}{3} - 4 \times \frac{1}{3} = 0$ by \ref{IN3V4}, \ref{4V} and \ref{5PV}. If $f$ has at least two common vertices with the outer cycle $\mathcal{C}$, and it is not incident with any $2$-vertex, then it receives at least $\frac{1}{6}$ from each incident vertex on the outer cycle $\mathcal{C}$, and then $\mu'(f) \geq 5 - 4 + 2\times \frac{1}{6} - 3 \times \frac{1}{3} = \frac{1}{3}$ by \ref{IN3V4}, \ref{3V}, \ref{4V} and  \ref{5PV}. If $f$ is incident with exactly one $2$-vertex, then $\mu'(f) \geq 5 - 4 + 2 \times \frac{1}{6} - \frac{2}{3} - 2 \times \frac{1}{3} = 0$ by \ref{IN3V4}, \ref{3V}, \ref{4V} and  \ref{5PV}. By \autoref{LT}\ref{Lg}, $f$ cannot be incident with more than one $2$-vertex. 

Let $f$ be an internal $6$-face. Note that there are no $4^{-}$-faces adjacent to $6$-faces, it follows that $f$ sends at most $\frac{1}{3}$ to each incident vertex, thus $\mu'(f) \geq 6 - 4 - 6 \times \frac{1}{3} = 0$. 

Let $f= [x_{1}x_{2}\dots x_{6}]$ be a $6$-face in $\mathcal{N}$. (i) Suppose that $f$ is incident with three $2$-vertices and an internal vertex $x_{6}$. If $x_{6}$ has at least three neighbors on the outer cycle $\mathcal{C}$, then $\mathcal{C}$ is an abnormal cycle (see \autoref{abnormala}, \autoref{abnormalb} and \autoref{abnormalc}), a contradiction. So $x_{6}$ has a neighbor not on $\mathcal{C}$, and $\mathcal{C}'=(\mathcal{C} - \{x_{2}, x_{3}, x_{4}\}) \cup \{x_{1}x_{6}, x_{5}x_{6}\}$ is a separating $10^{-}$-cycle, a contradiction. (ii) Suppose that $f$ is incident with one or two $2$-vertices. Thus, $f$ receives at least $\frac{1}{6}$ from each incident $3^{+}$-vertex on the outer cycle $\mathcal{C}$, and sends at most $\frac{1}{3}$ to each incident internal vertex, which implies that $\mu'(f) \geq 6 - 4 + 2 \times \frac{1}{6} - 2 \times \frac{2}{3} - 2 \times \frac{1}{3} = \frac{1}{3}$. (iii) If $f$ is not incident with any $2$-vertex, then it sends at most $\frac{1}{3}$ to each incident internal vertex, but does not send any charge to the incident vertices on the outer cycle $\mathcal{C}$, thus $\mu'(f) \geq 6 - 4 - 5 \times \frac{1}{3} > 0$. 

Let $f = [x_{1}x_{2}\dots x_{7}]$ be a $7$-face. Suppose that $f$ is an internal face. Since there are no $3$-faces adjacent to $7$-faces, we have that $f$ sends at most $\frac{4}{9}$ to each incident internal vertex. Note that seven is an odd number, $f$ is incident with at most six $3$-vertices which are incident with internal $4$-faces, thus $\mu'(f) \geq 7 - 4 - 6 \times \frac{4}{9} - \frac{1}{3} = 0$. Now, assume that $f$ is a face in $\mathcal{N}$. If $f$ is not incident with any $2$-vertex, then $\mu'(f) \geq 7 - 4 - 6 \times \frac{1}{2} = 0$. If $f$ is incident with one or two $2$-vertices, then $\mu'(f) \geq 7 - 4 - 2 \times \frac{2}{3} - 3 \times \frac{1}{2} > 0$. Consider the case $f$ is incident with exactly three $2$-vertices. It is observed that the three $2$-vertices are consecutive on the outer cycle, say $x_{2}, x_{3}$ and $x_{4}$. By the discharging rules, $f$ sends at most $\frac{1}{2}$ to each of $x_{6}$ and $x_{7}$. It follows that $\mu'(f) \geq 7 - 4 - 3 \times \frac{2}{3} - 2 \times \frac{1}{2} = 0$. Furthermore, $\mu'(f) = 0$ only if both $x_{6}$ and $x_{7}$ are $3$-vertices, and each of which is incident with a $4^{*}$-face, which implies that $\mathcal{C}$ is abnormal (see \autoref{abnormalf}), a contradiction. The final case: $f$ is incident with four $2$-vertices, say $x_{2}, x_{3}, x_{4}$ and $x_{5}$. Then it is incident with exactly one internal vertex $x_{7}$. Note that $x_{7}$ has degree at least three. If $x_{7}$ has another neighbor on the outer cycle other than $x_{1}$ and $x_{6}$, then $\mathcal{C}$ is abnormal (see \autoref{abnormalb} and \autoref{abnormale}), a contradiction. If $x_{7}$ has a neighbor not on the outer cycle, then $(\mathcal{C}\setminus \{x_{2}, x_{3}, x_{4}, x_{5}\}) \cup \{x_{1}x_{7}, x_{6}x_{7}\}$ is a separating normal $9^{-}$-cycle, a contradiction. In summary, every $7$-face has nonnegative final charge, and each $7$-face adjacent to $D$ has positive final charge. 

Let $f$ be an $8$-face. By \autoref{ADJF}, there are no $3$-faces adjacent to $8$-faces, so $f$ sends at most $\frac{1}{2}$ to each incident internal vertex. If $f$ is an internal face, then $\mu'(f) \geq 8 - 4 - 8 \times \frac{1}{2} = 0$. So we may assume that $f$ is an $8$-face in $\mathcal{N}$. If $f$ is not incident with any $2$-vertex, then $\mu'(f) \geq 8 - 4 - 7 \times \frac{1}{2} = \frac{1}{2}$. Suppose that $f$ is incident with at least one $2$-vertex. It follows that $f$ has at least two common $3^{+}$-vertices with the outer cycle $\mathcal{C}$, and it sends nothing to these vertices. Thus, $f$ sends $\frac{2}{3}$ to each incident $2$-vertex, and sends at most $\frac{1}{2}$ to each incident internal vertex, which implies that $\mu'(f) \geq 8 - 4 - 5 \times \frac{2}{3} - \frac{1}{2} > 0$. 

Let $f$ be a $k$-face with $k \geq 9$. By the discharging rules, it is easy to show the following fact.

\begin{enumerate}[label = \textbf{Fact-1.}, leftmargin = 9ex]
\item\label{Fact-1} The face $f$ sends nothing to the $3^{+}$-vertices on the outer cycle $\mathcal{C}$, and sends at most $\frac{1}{3}$ to each incident internal $4$-vertex, and sends at most $\frac{2}{3}$ to any vertex. 
\end{enumerate}

If $f$ is incident with some $2$-vertices, then $f$ is incident with at least two $3^{+}$-vertices on the outer cycle $\mathcal{C}$ and it sends nothing to these vertices, which implies that 
\begin{equation}\label{EQ5}
\mu'(f) \geq k - 4 - (k - 2) \times \frac{2}{3} = \frac{1}{3}(k - 8) > 0. 
\end{equation}
So we may assume that $f$ is not incident with any $2$-vertex. 

Let $\alpha$ be the number of incident bad vertices, $\beta$ be the number of incident light vertices, and let $\rho$ be the number of incident internal $3$-vertices. It is observed that $\alpha + \beta \leq \rho$. By \autoref{SIMITETRAD}, we can easily show the following fact on the parameters $\alpha$ and $\beta$. 

\begin{enumerate}[label = \textbf{Fact-2.}, leftmargin = 9ex]
\item\label{Fact-2} If $\alpha \geq 3$, then $\alpha + \beta \leq \rho \leq k - 2$. 
\end{enumerate}

If $\alpha \leq 3$, then $\mu'(f) \geq k - 4 - 3 \times \frac{2}{3} - (k - 3) \times \frac{1}{2} = \frac{k-9}{2} \geq 0$. If $\alpha \geq 4$ and $k \geq 10$, then $\mu'(f) \geq k - 4 - (k - 2) \times \frac{2}{3} - 2 \times \frac{1}{3} = \frac{1}{3}(k - 10) \geq 0$. It remains to assume that $\alpha \geq 4$, $k = 9$ and $f = [w_{1}w_{2}\dots w_{9}]$. 

Suppose that $\alpha = 7$. By \autoref{BAD}, the two non-bad vertices divide the bad vertices on $f$ into two parts, one consisting of three consecutive bad vertices and the other consisting of four consecutive bad vertices. Without loss of generality, we may assume that none of $w_{1}$ and $w_{6}$ is a bad vertex. By \autoref{SIMITETRAD}, $w_{1}w_{2}$, $w_{3}w_{4}$ and $w_{5}w_{6}$ are incident with $3$-faces. By symmetry, we may further assume that $w_{6}w_{7}$ and $w_{8}w_{9}$ are incident with $3$-faces. By \autoref{SIMITETRAD}, $w_{1}$ cannot be an internal $3$-vertex. If $w_{1}$ is an internal $5^{+}$-vertex or on the outer cycle $\mathcal{C}$, then $f$ sends nothing to $w_{1}$. If $w_{1}$ is an internal $4$-vertex, then $w_{1}$ is incident with three $9^{+}$-faces, and it receives nothing from $f$ by \ref{IN4V2}. Thus, $f$ sends nothing to $w_{1}$ in all cases, which implies that $\mu'(f) \geq 9 - 4 - 7 \times \frac{2}{3} - \frac{1}{3} = 0$. 

If $4 \leq \alpha \leq 5$, then $\mu'(f) \geq 9 - 4 - \alpha \times \frac{2}{3} - (\rho - \alpha) \times \frac{1}{2} - (9 - \rho) \times \frac{1}{3} = 2 - \frac{\alpha + \rho}{6} \geq 2 - \frac{5 + 7}{6} = 0$. It remains to assume that $\alpha = 6$. Recall that $f$ is not incident with any $2$-vertex. If $f$ has a common vertex with the outer cycle $\mathcal{C}$, then $f$ sends nothing to this vertex and $\mu'(f) \geq 9 - 4 - 6 \times \frac{2}{3} - 2 \times \frac{1}{2} = 0$. So we may assume that $f$ is an internal $9$-face. If $f$ is not incident with any light vertex, then $\mu'(f) \geq 9 - 4 - 6 \times \frac{2}{3} - 3 \times \frac{1}{3} = 0$. So we may assume that there is a light vertex on $f$. By the definitions of bad vertices and light vertices, a light vertex cannot be adjacent to two bad vertices on $f$, thus a light vertex must be adjacent to a non-bad vertex on $f$, which implies that the bad vertices on $f$ are divided into two parts by \autoref{BAD}. Moreover, each part has at most four consecutive bad vertices, thus each part has at least two consecutive bad vertices. Without loss of generality, we may assume that $w_{1}$ is a light vertex and $w_{9}$ is a non-bad vertex. Then $w_{2}$ and $w_{8}$ must be bad vertices. Since bad vertex and light vertex cannot be incident with the same $3$-face, $w_{2}w_{3}$ is incident with a $3$-face and $w_{1}w_{9}$ is incident with a $4^{-}$-face. By \autoref{SIMITETRAD} and \autoref{BAD}, $w_{3}$ is also a bad vertex and $w_{9}$ is not an internal $3$-vertex. Hence, $w_{9}$ must be an internal $4^{+}$-vertex. Since $w_{8}$ is a bad vertex, either $w_{7}w_{8}$ or $w_{8}w_{9}$ is incident with a $3$-face. If $w_{7}w_{8}$ is incident with a $3$-face, then $f$ sends nothing to $w_{9}$, which implies that $\mu'(f) \geq 9 - 4 - 6 \times \frac{2}{3} - 2 \times \frac{1}{2} = 0$. In the other case, $w_{8}w_{9}$ is incident with a $3$-face. Then $w_{6}w_{7}$ is incident with a $3$-face. By \autoref{SIMITETRAD}, one vertex in $\{w_{4}, w_{5}\}$ is an internal $4^{+}$-vertex, and the other is a bad $3$-vertex, thus $w_{4}w_{5}$ is incident with a $3$-face. Whenever $w_{4}$ or $w_{5}$ is an internal $4^{+}$-vertex, it receives nothing from $f$ by \ref{IN4V2}, which implies that $\mu'(f) \geq 9 - 4 - 6 \times \frac{2}{3} - \frac{1}{2} - \frac{1}{3} > 0$. \end{proof}

\begin{lemma}
Every vertex has nonnegative final charge. 
\end{lemma}
\begin{proof}
If $v$ is a $2$-vertex, then it is on the outer cycle, and it receives $\frac{2}{3}$ from the incident face in $\mathcal{N}$ and $\frac{4}{3}$ from the outer face $D$ by \ref{2V}, which implies that $\mu'(v) = 2 - 4 + \frac{2}{3} + \frac{4}{3} = 0$. Let $v$ be a $3$-vertex on the outer cycle. By the adjacency of the faces, $v$ is incident with a $4^{-}$-face and a $7^{+}$-face, or it is incident with two $5^{+}$-faces. Thus, $\mu'(v) \geq 3 - 4 + \frac{4}{3} - \max\left\{\frac{1}{3}, 2\times \frac{1}{6}\right\} = 0$ by \ref{3V}. If $v$ is a $4$-vertex on the outer cycle, then it receives $1$ from the outer face, sends $1$ to an incident special $4^{-}$-face and at most $\frac{1}{3}$ to each of the other incident $6^{-}$-face in $\mathcal{N}$ by \ref{4V}, which implies that $\mu'(v) \geq 4 - 4 + 1 - \max\{1, 3 \times \frac{1}{3}\} = 0$. If $v$ is a $5^{+}$-vertex on the outer cycle, then it receives $1$ from the outer face, and averagely sends at most $\frac{1}{2}$ to each incident face in $\mathcal{N}$, and then $\mu'(v) \geq \deg(v) - 4 + 1 - (\deg(v) - 1) \times \frac{1}{2} = \frac{\deg(v) - 5}{2}\geq 0$. 

If $v$ is a bad vertex, then $\mu'(v) = 3 - 4 + 2 \times \frac{2}{3} - \frac{1}{3}= 0$ by \ref{3F} and \ref{IN3V1}. If $v$ is incident with a $4^{*}$-face or a special $3$-face, then $\mu'(v) = 3 - 4 + 2 \times \frac{1}{2} = 0$ by \ref{IN3V2}. If $v$ is incident with an internal $4$-face, then $\mu'(v) = 3 - 4 + \frac{1}{9} + 2 \times \frac{4}{9} = 0$ by \ref{IN3V3}. If $v$ is a good vertex, then $\mu'(v) = 3 - 4 + 3 \times \frac{1}{3} = 0$ by \ref{IN3V4}. If $v$ is an internal $4$-vertex which is incident with two negative faces, then $\mu'(v) = 4 - 4 + 2\times \frac{1}{3} - 2 \times \frac{1}{3} = 0$ by \ref{3F} and \ref{IN4V}. If $v$ is an internal $4$-vertex which is incident with exactly one negative face and a $9^{+}$-face at the opposite side, then $\mu'(v) = 4 - 4 + \frac{1}{3} - \frac{1}{3} = 0$ by \ref{3F} and \ref{IN4V}. If $v$ is incident with exactly one negative face and an $8^{-}$-face at the opposite side, then $\mu'(v) = 4 - 4 + 2 \times \frac{1}{6} - \frac{1}{3} = 0$ by \ref{3F} and \ref{IN4V}. If $v$ is an internal $4$-vertex which is not incident with any negative face, then $\mu'(v) = 4 - 4 = 0$. Note that there are no adjacent $5^{-}$-faces, thus every vertex $v$ is incident with at most $\lfloor\frac{\deg(v)}{2}\rfloor$ small faces. If $v$ is an internal $5^{+}$-vertex, then it sends at most $\frac{1}{3}$ to each incident $5^{-}$-face, which implies that $\mu'(v) \geq \deg(v) - 4 - \frac{1}{3} \times \lfloor\frac{\deg(v)}{2}\rfloor > 0$. 
\end{proof}

\begin{lemma}
The outer face $D$ has nonnegative charge, and there exists an element in $V(G) \cup F(G)$ having positive final charge. 
\end{lemma}

By \ref{2V}--\ref{5PV}, the outer face $D$ sends $\frac{4}{3}$ to each incident $3^{-}$-vertex, and sends $1$ to each incident $4^{+}$-vertex, thus $\mu'(D) \geq |D| + 4 - \frac{4}{3}|D| = \frac{1}{3}(12 - |D|) \geq 0$. Therefore, every element in $V(G) \cup F(G)$ has a nonnegative final charge. In particular, $\mu'(D) = 0$ holds if and only if $|D| = 12$ and each vertex incident with $D$ receives $\frac{4}{3}$ from $D$. So we may assume that $|D| = 12$ and each vertex on the outer cycle $\mathcal{C}$ is a $3^{-}$-vertex.

Let $f$ be an arbitrary $k$-face adjacent to $D$. By the discharging rules, $f$ sends nothing to $3^{+}$-vertices on the outer cycle, and sends at most $\frac{2}{3}$ to each of the other incident vertex. If $k \geq 9$, then $\mu'(f) \geq k - 4 - (k - 2) \times \frac{2}{3} > 0$. Recall that every $6$-, $7$- and $8$-face adjacent to $D$ has a positive final charge, thus $D$ is not adjacent to any $6^{+}$-face. Therefore, there exists a $3$-vertex on the outer cycle such that it is incident with two adjacent $5^{-}$-faces, a contradiction. 
\end{proof}

\section{\textit{IF}-coloring}
\label{sec:3}
For \textit{IF}-coloring, a vertex is \textbf{colored with $I$} if the image is $I$ in a mapping, or it is in the part $I$ of an $(I, F)$-partition. An \textbf{$F$-path} is a path whose vertices are all colored with $F$, and an \textbf{$F$-cycle} is a cycle whose vertices are all colored with $F$. 

Given a graph $G$ and a cycle $C$ in $G$, an \textit{IF}-coloring $\phi_{C}$ of $G[V(C)]$ can \textbf{be superextended} to $G$ if there exists an \textit{IF}-coloring $\phi_{G}$ of $G$ that extends $\phi_{C}$ with the property that there are no $F$-paths (having at least one vertex not on $C$) linking two vertices of $C$. We say that $C$ is \textbf{superextendable} to $G$ if every \textit{IF}-coloring of $G[V(C)]$ can be superextended to $G$. For convenience, we also say that a vertex $w$ is \textbf{superextendable} to $G$ if every \textit{IF}-coloring of $w$ can be extended to an \textit{IF}-coloring of $G$.

For an \textit{IF}-coloring, a vertex $v$ is \textbf{$F$-reachable} to a cycle if there is a path from $v$ to a vertex on the cycle such that all the vertices on the path are colored with $F$. 

Instead of proving \autoref{IFMR}, we prove the following stronger result. 
\begin{restatable}{theorem}{MRIF}\label{MRIF}
Let $G$ be a graph in $\mathscr{G}$. Let $S$ be a set of vertices of $G$ such that $|S| \leq 1$ or $S$ consists of all vertices on a normal cycle of $G$. If $|S| \leq 12$, then every \textit{IF}-coloring $\phi$ of $G[S]$ can be superextended to an \textit{IF}-coloring of $G$. 
\end{restatable}

\begin{remark}
Analogously, not every \textit{IF}-coloring of the $11$- or $12$-cycle can be superextended to the whole graph. Thus, we require the condition that $S$ consists of all vertices on a ``normal'' cycle.
\end{remark}

The following result is a direct consequence of \autoref{MRIF}. 
\begin{theorem}\label{MRIF-}
Every graph in $\mathscr{G}$ is \textit{IF}-colorable. 
\end{theorem}

\begin{proof}[Sketch of a proof for \autoref{MRIF}]
Suppose that $G$ is a minimal counterexample to \autoref{MRIF}. That is, there exists an \textit{IF}-coloring of $G[S]$ that cannot be superextended to an \textit{IF}-coloring of $G$ such that $|V(G)| + |E(G)| - |S|$ is minimized. 

\begin{lemma}\label{IF3-8}
Every $8^{-}$-cycle has no chords. 
\end{lemma}

\begin{lemma}\label{IFLT}
\text{}
\begin{enumerate}[label = (\alph*)]
\item\label{IFLa} $S \neq V(G)$. 
\item\label{IFLb} $G$ is $2$-connected, and thus the boundary of every face is a cycle. 
\item\label{IFLc} Each vertex not in $S$ has degree at least three. 
\item\label{IFLd} Either $|S| = 1$ or $G[S]$ is an induced cycle of $G$. 
\item\label{IFLe} There are no separating normal $k$-cycles for $3 \leq k \leq 12$. Thus, every edge on an abnormal cycle is incident with a $4$-, $5$-, $6$- or $7$-face. 
\item\label{IFLf} $G[S]$ is an induced cycle of $G$. For convenience, we can redraw the graph $G$ such that $G[S]$ is the outer cycle $\mathcal{C}$ of $G$. Let $D$ be the outer face which is bounded by $G[S]$. 
\item\label{IFLg} Every $5$-face ($\neq$ outer face) is incident with at most one $2$-vertex. 
\end{enumerate}
\end{lemma}

\begin{proof}
The proof is similar to the one in \autoref{LT}, so the reader can do it as an exercise, or find it in  the arXiv version. 
\end{proof}

\begin{lemma}\label{IFADJACENTFACE}
There are no $3$-faces adjacent to $8^{-}$-faces, no $4$-faces adjacent to $6^{-}$-faces, and no adjacent $5$-faces. 
\end{lemma}

\begin{proof}
The proof is the same with that in \autoref{3-8}. 
\end{proof}

\begin{lemma}\label{IFTETRAD}
Let $w_{0}, w_{1}, w_{2}, w_{3}$ and $w_{4}$ be five consecutive vertices on a $5^{+}$-face. If $w_{1}, w_{2}, w_{3}$ and $w_{4}$ are all $3$-vertices, and $w_{1}w_{2}$ is incident with a $3$-face $ww_{1}w_{2}$, then at least one vertex in $\{w_{1}, w_{2}, w_{3}, w_{4}\}$ is on the outer cycle $\mathcal{C}$. 
\end{lemma}

\begin{proof}
Suppose to the contrary that none of $\{w_{1}, w_{2}, w_{3}, w_{4}\}$ is on the outer cycle $\mathcal{C}$. Let $w'$ be the neighbor of $w_{3}$ other than $w_{2}, w_{4}$, and let $G^{*} = G - \{w_{1}, w_{2}, w_{3}, w_{4}\}$. It is observed that $w_{0}, w_{1}, w_{2}, w_{3}, w_{4}, w$ and $w'$ are seven distinct vertices. Similar to \autoref{SIMITETRAD}, we can prove that the distance between $w_{0}$ and $w'$ is at least nine in $G^{*}$. 

Let $G'$ be the graph obtained from $G^{*}$ by identifying $w_{0}$ and $w'$. Since the distance between $w_{0}$ and $w'$ is at least nine in $G^{*}$, the graph $G'$ has no loop, no multiple edge and no new $8^{-}$-cycle, thus $G'$ is a simple plane graph satisfying the assumption of \autoref{MRIF}. Moreover, $\mathcal{C}$ is also a normal cycle of $G'$ and it has no chords in $G'$. This implies that $\phi$ is an \textit{IF}-coloring of $G'[S]$. Since $|V(G')| < |V(G)|$, the \textit{IF}-coloring $\phi$ of $G'[S]$ can be superextended to an \textit{IF}-coloring $\varphi$ of $G'$. We color $w_{0}$ and $w'$ with the same color as the new vertex in $G'$. If one neighbor of $w_{4}$ is colored with $I$, then we color $w_{4}$ with $F$; otherwise, we color $w_{4}$ with $I$. 

If $\varphi(w_{0}) = \varphi(w') = I$, then let $\varphi(w_{1}) = \varphi(w_{3}) = F$ and $\varphi(w_{2}) \neq \varphi(w)$. So we may assume that $\varphi(w_{0}) = \varphi(w') = F$, and let $\varphi(w_{3}) \neq \varphi(w_{4})$. If $\varphi(w) = I$, then let $\varphi(w_{1}) = \varphi(w_{2}) = F$. So we may assume that $\varphi(w) = F$. If $\varphi(w_{3}) = I$, then let $\varphi(w_{1}) = I$ and $\varphi(w_{2}) = F$. In the final, we may assume that $\varphi(w_{0}) = \varphi(w) = \varphi(w') = \varphi(w_{3}) = F$. It is observed that there are no $F$-paths in $H$ between $w'$ and $w_{0}$. Meanwhile, at least one of $w'$ and $w_{0}$ is not reachable to $\mathcal{C}$. When we color $w_{1}$ with $F$, it is a superextension to $G - w_{2}$, we can color $w_{2}$ with $I$. Otherwise, there is an $F$-path in $H$ between $w$ and $w_{0}$, or both $w$ and $w_{0}$ are reachable to $\mathcal{C}$. It follows that there are no $F$-paths between $w$ and $w'$, and at least one of $w$ and $w'$ is not reachable to $\mathcal{C}$. Then let $\varphi(w_{1}) = I$ and $\varphi(w_{2}) = F$. 

It is easy to check that the resulting coloring is always a superextension to $G$, a contradiction. 
\end{proof}

Let $w$ be a vertex on the outer cycle $\mathcal{C}$, and let $w_{1}, w_{2}, \dots, w_{k}$ be consecutive neighbors in a cyclic order. If $f$ is a face in $\mathcal{N}$ incident with $ww_{i}$ and $ww_{i+1}$, but neither $ww_{i}$ nor $ww_{i+1}$ is an edge of $\mathcal{C}$, then we call $f$ a \textbf{special} face (at $w$). A $4$-face is a \textbf{$4^{*}$-face} if it has three common vertices with $\mathcal{C}$. An internal $3$-vertex is \textbf{bad} if it is incident with a $3$-face which is not special, \textbf{light} if it is incident with an internal $4$-face or a $4^{*}$-face or a special $3$-face, \textbf{good} if it is neither bad nor light. According to \autoref{IFTETRAD}, we have the following result on bad vertices. 

\begin{lemma}\label{IFBAD}
There are no five consecutive bad vertices on the boundary of a $5^{+}$-face. 
\end{lemma}

\begin{lemma}\label{IF4F}
If a $4$-face in $\mathcal{N}$ has exactly two common vertices with $\mathcal{C}$, then these two vertices are consecutive on the $4$-face. 
\end{lemma}

\begin{proof}
Suppose that $f$ is a $4$-face in $\mathcal{N}$ that has exactly two common vertices with $\mathcal{C}$. If these two vertices are not consecutive on the $4$-face, then there exists a separating normal $8^{-}$-cycle, a contradiction. 
\end{proof}

Assume that $f = [v_{1}v_{2}\dots v_{l}]$ is an internal $(3, 3, 3, 3)$-face or an internal $(3, 3, 3, 3, 3)$-face. Let $u_{i}$ be the third neighbor of $v_{i}$ for $1 \leq i \leq l$. Note that every $8^{-}$-cycle has no chords, and there are no separating $4$- or $5$-cycles. It is observed that $\{u_{i}, v_{i}\mid 1 \leq i \leq l\}$ contains $2l$ distinct vertices. Let $\Gamma$ be the graph $G \setminus (V(f) \setminus \{v_{3}\})$, and let $G^{*}$ be the graph obtained from $\Gamma$ by identifying $u_{1}$ and $v_{3}$ into a new vertex $z$. 

\begin{lemma}
The graph $G^{*}$ is in the class $\mathscr{G}$, and $\mathcal{C}$ is an induced cycle of $G^{*}$. 
\end{lemma}
\begin{proof}
The proof is the same with that in \autoref{CLOSED}. 
\end{proof}

\begin{lemma}
There are no internal $(3, 3, 3, 3)$-faces or $(3, 3, 3, 3, 3)$-face. 
\end{lemma}
\begin{proof}
We know that $(G^{*}, \mathcal{C}, \phi)$ satisfies all the requirements in \autoref{MRIF}. By the minimality of $G$, the \textit{IF}-coloring $\phi$ can be superextended to an \textit{IF}-coloring $\varphi$ of $G^{*}$. We color $u_{1}$ and $v_{3}$ with the same color as the new vertex by the identification. 

$\bullet$ $\bm{l = 5}$. 

Assume that $\varphi(u_{1}) = \varphi(v_{3}) = I$. Then let $\varphi(v_{1}) = \varphi(v_{2}) = \varphi(v_{4}) = F$. Furthermore, we color $v_{5}$ such that $\varphi(v_{5}) \neq \varphi(u_{5})$, this is invalid only if $u_{2}v_{2}v_{1}v_{5}v_{4}u_{4}$ is on an $F$-cycle or is on an $F$-path linking two vertices of $\mathcal{C}$. In both cases, we recolor $v_{2}$ and $v_{4}$ with $I$, and $v_{3}$ with $F$. 

Assume that $\varphi(u_{1}) = \varphi(v_{3}) = F$. If $\varphi(u_{4}) = I$, then let $\varphi(v_{1}) = \varphi(v_{4}) = F$, $\varphi(v_{2}) \neq \varphi(u_{2})$ and $\varphi(v_{5}) \neq \varphi(u_{5})$, this is invalid only if $v_{1}, v_{2}, v_{3}, v_{4}, v_{5}$ are all colored with $F$. In this case, we recolor $v_{1}$ with $I$. So we may assume that $\varphi(u_{4}) = F$. If $\varphi(u_{2}) = I$ or $\varphi(u_{3}) = I$, then let $\varphi(v_{1}) = \varphi(v_{4}) = I$ and $\varphi(v_{2}) = \varphi(v_{5}) = F$. So we may assume that $\varphi(u_{1}) = \varphi(u_{2}) = \varphi(u_{3}) = \varphi(u_{4}) = F$. If $\varphi(u_{5}) = I$, then let $\varphi(v_{2}) = \varphi(v_{4}) = I$, and $\varphi(v_{1}) = \varphi(v_{5}) = F$. Now, we may assume that all the vertices $u_{i}$ are colored with $F$. Then let $\varphi(v_{3}) = \varphi(v_{5}) = I$ and $\varphi(v_{1}) = \varphi(v_{2}) = \varphi(v_{4}) = F$, this is invalid only if $u_{1}v_{1}v_{2}u_{2}$ is on an $F$-cycle or an $F$-path linking two vertices of $\mathcal{C}$. In this case, let $\varphi(v_{1}) = \varphi(v_{4}) = I$ and $\varphi(v_{2}) = \varphi(v_{3}) = \varphi(v_{5}) = F$. 

$\bullet$ $\bm{l = 4}$.

Assume that $\varphi(u_{1}) = \varphi(v_{3}) = I$. Then let $\varphi(v_{1}) = \varphi(v_{2}) = \varphi(v_{4}) = F$, this is invalid only if $u_{2}v_{2}v_{1}v_{4}u_{4}$ is on an $F$-cycle or is on an $F$-path linking two vertices of $\mathcal{C}$. In both cases, let $\varphi(v_{1}) = \varphi(v_{3}) = F$, $\varphi(v_{2}) = \varphi(v_{4}) = I$. 

Assume that $\varphi(u_{1}) = \varphi(v_{3}) = F$. Then let $\varphi(v_{1}) = F$, $\varphi(v_{2}) \neq \varphi(u_{2})$ and $\varphi(v_{4}) \neq \varphi(u_{4})$, this is invalid only if $\varphi(u_{2}) = \varphi(u_{4}) = I$. In this case, we recolor $v_{1}$ with $I$. 
\end{proof}
Note that \autoref{full} is only prepared for \autoref{SIMITETRAD}, so we do not need it here. All the structural lemmas are the same in the proof processes, so the discharging part is the same with that in \autoref{MRIF}. 
\end{proof}

\section{Final discussion}
\label{sec:4}
If we can relax the normally adjacent $5$-cycles in \autoref{MR}, then \autoref{MR} implies that every planar graph without $3$-, $6$-, $7$-cycles is DP-$3$-colorable. On the other hand, if we can allow the normally adjacent $4$-cycles in \autoref{MR}, then \autoref{MR} implies that every planar graph without $3$-, $7$-, $8$-cycles is DP-$3$-colorable. But until now, we don't know whether every planar graph without $3$-, $6$-, $7$-cycles is DP-$3$-colorable or not, and we don't know whether every planar graph without $3$-, $7$-, $8$-cycles is DP-$3$-colorable or not. It seems that it is not easy to solve these two problems. 

Dvo\v{r}\'{a}k, Lidick\'{y}, and \v{S}krekovski \cite{MR2680225} proved that every planar graph without $3$-, $6$-, $7$-cycles is $3$-choosable. The same authors \cite{MR2552620} also proved that every planar graph without $3$-, $7$-, $8$-cycles is $3$-choosable. So it is interesting to consider the following problems. 

\begin{problem}\label{Problem1}
Is every planar graph without $3$-, $6$-, $7$-cycles DP-$3$-colorable?
\end{problem}

\begin{problem}
Is every planar graph without $3$-, $7$-, $8$-cycles DP-$3$-colorable?
\end{problem}

\begin{problem}
Does every planar graph without $3$-, $6$-, $7$-cycles have an \textit{IF}-coloring?
\end{problem}

\begin{problem}
Does every planar graph without $3$-, $7$-, $8$-cycles have an \textit{IF}-coloring?
\end{problem}

\noindent\textbf{Added Note.} Recently, Han \etal \cite{Han2022} proved that every triangle-free planar graph without 4-cycles normally adjacent to 4- and 5-cycles is DP-$3$-colorable. This solves \autoref{Problem1}.

\vskip 0mm \vspace{0.3cm} \noindent\textbf{Acknowledgments.} This work was supported by the Fundamental Research Funds for Universities in Henan (YQPY20140051).

\appendix
\section*{Appendix}
\begin{proof}[Proof of \autoref{IFLT}]
\ref{IFLa} Suppose to the contrary that $S = V(G)$. Every \textit{IF}-coloring of $G[S]$ is an \textit{IF}-coloring of $G$, a contradiction. 

\ref{IFLb} It is observed that $G$ is connected. Suppose to the contrary that $G$ has a cut-vertex $w$. We may assume that $G = G_{1} \cup G_{2}$ and $G_{1} \cap G_{2} = \{w\}$. By the assumption of the set $S$, either $S \subseteq V(G_{1})$ or $S \subseteq V(G_{2})$. We may assume that $S \subseteq V(G_{1})$. By the minimality of $G$, the \textit{IF}-coloring $\phi$ of $G[S]$ can be superextended to an \textit{IF}-coloring $\phi_{1}$ of $G_{1}$, and $\phi_{1}(w)$ can be superextended to an \textit{IF}-coloring $\phi_{2}$ of $G_{2}$. These two colorings $\phi_{1}$ and $\phi_{2}$ together give an \textit{IF}-coloring of $G$ whose restriction on $G[S]$ is $\phi$, a contradiction. 

\ref{IFLc} Suppose that there exists a vertex $w$ not in $S$ having degree two. By the minimality of $G$, the \textit{IF}-coloring of $G[S]$ can be superextended to an \textit{IF}-coloring of $G - w$. If the two neighbors of $w$ are colored with $F$, then we color $w$ with $I$; otherwise, we color $w$ with $F$. 

\ref{IFLd} If $S = \emptyset$, then we put any vertex into $S$ to make $|S| = 1$. Suppose that $S = V(\mathcal{Q})$ and $\mathcal{Q}$ is a cycle with a chord $uv$. It is observed that the \textit{IF}-coloring of $G[S]$ is also an \textit{IF}-coloring of the induced subgraph in $G - uv$. By the minimality of $G$, the \textit{IF}-coloring $\phi$ of $G[S]$ can be superextended to an \textit{IF}-coloring of $G - uv$, and hence it is also a superextension of $G$, a contradiction. 

\ref{IFLe} We first show that $G[S]$ cannot be a separating cycle. Suppose to the contrary that $G[S]$ is a separating (normal) cycle $\mathcal{O}$. By the minimality of $G$, the \textit{IF}-coloring $\phi$ of $\mathcal{O}$ can be superextended to an \textit{IF}-coloring $\phi_{1}$ of $\Int(\mathcal{O})$, and another \textit{IF}-coloring $\phi_{2}$ of $\Ext(\mathcal{O})$. These two colorings $\phi_{1}$ and $\phi_{2}$ together give a superextension, a contradiction. 

Thus, either $|S| = 1$ or $S$ consists of all vertices on a face of $G$. Let $\mathcal{Q}$ be a separating normal $k$-cycle with $3 \leq k \leq 12$. Thus, we may assume that $S \subseteq \Ext(\mathcal{Q})$. By the minimality of $G$, the \textit{IF}-coloring $\phi$ of $G[S]$ can be superextended to an \textit{IF}-coloring $\varphi_{1}$ of $\Ext(\mathcal{Q})$. Similarly, the restriction of $\varphi_{1}$ on $\mathcal{Q}$ can be superextended to an \textit{IF}-coloring $\varphi_{2}$ of $\Int(\mathcal{Q})$. These two colorings $\varphi_{1}$ and $\varphi_{2}$ together give a superextension of $G$, a contradiction. 

\ref{IFLf} According to \ref{IFLd}, suppose to the contrary that $S = \{w\}$. We first assume that $w$ is on a $10^{-}$-cycle $\mathcal{Q}$. Without loss of generality, we may assume that $\mathcal{Q}$ is a shortest cycle containing $w$. Then $\mathcal{Q}$ is an induced cycle. By \ref{IFLe}, we may assume that $\ext(\mathcal{Q}) = \emptyset$ and $\mathcal{Q}$ is the outer cycle. By \ref{IFLc} and $\mathcal{Q}$ is an induced cycle, every vertex on $\mathcal{Q}$ other than $w$ has a neighbor in $\int(\mathcal{Q})$, which implies that $\int(\mathcal{Q}) \neq \emptyset$. By the minimality of $G$, the \textit{IF}-coloring $\phi$ of $\{w\}$ can be superextended to an \textit{IF}-coloring $\phi_{1}$ of $\mathcal{Q}$. By the minimality of $G$, the \textit{IF}-coloring $\phi_{1}$ of $\mathcal{Q}$ can be further superextended to an \textit{IF}-coloring of $G$, a contradiction. 

So we may assume that every cycle containing $w$ has length at least $11$. Let $w$ be incident with a face $[w_{1}ww_{2}\dots w_{1}]$. Let $G'$ be obtained from $G$ by adding a chord $w_{1}w_{2}$ in the face, let $S' = \{w, w_{1}, w_{2}\}$. We can easily check that $G'$ is a plane graph satisfying the assumption of \autoref{MRIF}. By the minimality of $G$, the \textit{IF}-coloring $\phi$ of $\{w\}$ can be superextended to an \textit{IF}-coloring $\phi_{1}$ of $G'[S']$. By the minimality of $G$, the \textit{IF}-coloring $\phi_{1}$ of $G'[S']$ can be further superextended to an \textit{IF}-coloring $\varphi$ of $G'$. It is observed that $\varphi$ is an \textit{IF}-coloring of $G$, a contradiction. 

\ref{IFLg} Note that every $2$-vertex and its two neighbors are all on the outer cycle. Suppose to the contrary that $f = [x_{1}x_{2}x_{3}x_{4}x_{5}]$ is a $5$-face which is incident with two $2$-vertices. Note that the two $2$-vertices must be adjacent on the outer cycle, say $x_{2}$ and $x_{3}$. It follows that $x_{1}$ and $x_{4}$ are on the outer cycle $\mathcal{C}$. If $x_{5}$ has three neighbors on $\mathcal{C}$, then $\mathcal{C}$ is abnormal (see \autoref{abnormala} and \autoref{abnormalb}), a contradiction. Thus, $x_{5}$ has a neighbor not on the outer cycle $\mathcal{C}$, and $\mathcal{C}'= (\mathcal{C} - \{x_{2}, x_{3}\}) \cup \{x_{1}x_{5}, x_{4}x_{5}\}$ is a separating $11^{-}$-cycle. By \autoref{IFLT}\ref{IFLe}, $\mathcal{C}'$ is an abnormal $11$-cycle (see \autoref{abnormala}). It follows that $\mathcal{C}$ is an abnormal $12$-cycle (see \autoref{abnormald}), a contradiction. 
\end{proof}

\end{document}